\documentclass[12pt,a4paper,oneside,reqno]{amsart}
\ifdefined\SMART
\usepackage[paperwidth=12cm,paperheight=24cm,left=5mm,right=5mm,top=10mm,bottom=5mm]{geometry}
\else
\calclayout
\fi

\usepackage{color,amssymb,latexsym,amsfonts,textcomp,fancyhdr, calc,graphicx}
\usepackage{pdfpages,amsfonts,amsthm,amssymb,latexsym,graphicx,leftidx,fancyhdr}
\usepackage{amsmath,amstext,amsthm,amssymb,amsxtra}
\newtheoremstyle{nobracket}
{}{}{\itshape}{}{\bfseries}{.}{ }%
{\thmname{#1}\thmnumber{ #2}\thmnote{ #3}}
\usepackage{enumitem}
\theoremstyle{nobracket}
\usepackage[dvipsnames,svgnames,table,rgb]{xcolor}
\usepackage[allcolors=blue,allbordercolors=blue,pdfborderstyle={/S/U/W 1}]{hyperref}
\usepackage[ocgcolorlinks]{ocgx2}
\usepackage{dsfont}
\usepackage[framemethod=tikz]{mdframed}
\usepackage{asymptote}
\usepackage{enumerate}
\usepackage{graphicx}
\graphicspath{{images/}}
\usepackage{subcaption}
\usepackage{float}

\usepackage{txfonts,pxfonts,tikz} 
\usepackage{tgschola}
\usepackage[T1]{fontenc}
\usepackage{mathtools}

\newtheorem{theorem}{Theorem}[section]
\newtheorem{corollary}{Corollary}[section]
\newtheorem{lemma}{Lemma}[section]

\theoremstyle{definition}
\newtheorem{definition}{Definition}[section]

\newcommand{\beql}[1]{\begin{equation}\label{#1}}
	\newcommand{\eeq}{\end{equation}}
\newcommand{\comment}[1]{}

\newcommand{\Floor}[1]{{\left\lfloor{#1}\right\rfloor}}

\DeclareMathOperator{\rank}{rank}

\newcommand{\RR}{{\mathbb R}}

\newcommand{\ZZ}{{\mathbb Z}}
\newcommand{\NN}{{\mathbb N}}

\newcommand{\vol}{{\rm vol\,}}

\newcounter{rem}
\title[Bounded and Measurable common Fundamental domains]{Bounded and measurable common fundamental domains for two lattices}

\author{Emmanouil Spyridakis}
\address{\href{http://math.uoc.gr/en/index.html}{Department of Mathematics and Applied Mathematics}, University of Crete,\\Voutes Campus, 70013 Heraklion, Greece.}
\email{manos.ch.spyridakis@gmail.com}

\begin{document}
\begin{abstract}
    Suppose that $L, M$ are two full-rank lattices in Euclidean space with $\vol(L)=\vol(M)$. We give a new proof on the existence of a bounded and Lebesgue measurable set that tiles $\RR^d$ with both $L,M$ using the measurable Hall's Theorem \cite{CS22}. This proof is direct and does not go
    through the intermediate results on cut-and-project sets involved in the proof given in \cite{GK25}. We also show the existence of a bounded, set-theoretic (i.e., not necessarily measurable) common fundamental domain of $L,M$ assuming only that $\vol(L)=\vol(M)$. Combining these results we show the existence of a bounded and Lebesgue measurable common fundamental domain for any two full-rank lattices of equal volumes. Finally we show that a set-theoretic bounded, common fundamental domain cannot exist when $\vol(L)\neq \vol(M)$. 
\end{abstract}
\maketitle
\tableofcontents
\section{Introduction}

Recently, there has been a lot of work regarding the existence of a bounded and Lebesgue measurable common fundamental domain of two full-rank lattices in $\RR^d$ when their volumes are equal. In \cite{GK25}, Theorem 1 showed that for any two full-rank lattices of the same volume, there exists a bounded and measurable set that tiles $\RR^d$ with both lattices (a set $E$ tiles with a lattice $L$ if almost every point in Euclidean space can be written uniquely as
a sum of an element of $E$ and an element of $L$). The method used there involved cut-and-project sets to prove the special but decisive case when the sum of the two lattices is dense in Euclidean space. We will give a new proof for Theorem 1 in \cite{GK25} using the measurable Hall's Theorem \cite{CS22} in the cases where cut-and-project sets were involved in the proofs of \cite{GK25}. The measurable Hall's Theorem has lately seen several applications, one of which is in cut-and-project sets (see \cite{EGKL25}, section 4). This application made us believe that the measurable Hall's Theorem can be used to prove the results showed in \cite{GK25} bypassing the intermediate results about cut-and-project sets used there and providing a direct proof.

Furthermore, Theorem 3 in \cite{Kol97} gives us a sufficient condition for the existence of a Lebesgue measurable common fundamental domain $F$ for a finite collection of full-rank lattices. The conditions needed there, are that the lattices admit a common fundamental domain $F_1$ (no measurability is assumed) along with a measurable set that tiles the Euclidean space by translations of each lattice in the collection. In particular, from the proof of this theorem, one can easily conclude that $F\subseteq F_1 \cup F_2$ and so if $F_1, F_2$ are bounded sets, then so is $F$. Using this, along with the measurable Hall's Theorem we mentioned previously, we will show the following:

\begin{theorem}\label{BMF}
Any two full-rank lattices $L, M \subseteq \RR^d$, admitting a bounded common fundamental domain on $\RR^d$, also admit one that is bounded and measurable.
\end{theorem}

Notice that we do not assume that the two lattices have equal volume. However, as an immediate consequence of Theorem \ref{BMF}, we see that this is the only possible way; any two full-rank lattices of unequal volumes cannot admit even a non-measurable bounded fundamental domain (see Corollary  \ref{Evol}). This was partially proved in \cite{GKS25} Theorem 1.1, in the special case when the sum of the two lattices is direct.

Moreover, in \cite{Kol97} section 3.2, the author showed that for any finite collection of full-rank lattices in $\RR^d$ with equal volumes and whose sum is direct (i.e., they pairwise intersect at the origin) and dense in $\RR^d$, there exists a bounded, not necessarily measurable common fundamental domain. We extend this result for any two full-rank lattices as follows:

\begin{theorem}\label{BF}
    Any two full-rank lattices $L, M\subseteq \RR^d$ of the same volume, admit a bounded, not necessarily measurable, common fundamental domain on $\RR^d$.
\end{theorem}

Finally, the main result of this paper comes immediately after combining Theorem \ref{BF} and Theorem \ref{BMF}:

\begin{theorem}\label{MT}
    Any two full-rank lattices $L, M \subseteq \RR^d$ of the same volume admit a bounded, (Lebesgue) measurable common fundamental domain in $\RR^d$.
\end{theorem}

To formalize the discussion above, we next introduce the notation adopted throughout the paper.

\textbf{Notation}: A  lattice $L$ in $\RR^d$ is a discrete subgroup of $\RR^d$. The dimension of the $\RR$-linear subspace that is spanned by $L$ is called the \textit{rank} of $L$. In particular, if $L$ spans $\RR^d$, we will call it a full-rank lattice.
It is known that any full-rank lattice on $\RR^d$ is equal to $A\ZZ^d$ for some non-singular $d\times d$ matrix. The matrix $A$ represents a $\ZZ$-basis of $L$. Even though a basis of $L$ is not unique and so neither is the matrix $A$, what remains the same between the matrices that represent $L$ is the absolute value of their determinant and this number is called \textit{volume} of $L$ and we denote it by $\vol(L):= |\det(A)|$. A \textit{fundamental domain}  of $L=A\ZZ^d$ in $\RR^d$ is a set that contains exactly one representative from each coset of $L$ in $\RR^d$. If a fundamental domain is measurable then its measure is $\vol(L)$. (see Section \ref{BFD} for more details)

The structure of the rest of this paper is as follows.

In the preliminary Section $2$, we review fundamental domains, the notion of equidecomposability, and their relationship, ending up to stating measurable Hall's Theorem \cite{CS22}. 

In Section $3$ we prove Theorem \ref{BF} and in Section $4$ we prove  Theorem \ref{BMF} and the main result, Theorem \ref{MT}.
\section{Bounded Fundamental domains and equidecompositions}
In this section, we introduce the notion of equidecomposition and examine its relationship with the existence of a bounded fundamental domain for a full-rank lattice. We establish several preliminary lemmas which, although technically straightforward, play a crucial role in the proofs of our theorems. Finally, we close this section by stating the measurable Hall's Theorem written  by Cieśla and
Sabok \cite{CS22}, which will be used later.
\subsection{Bounded Fundamental domains}\label{BFD}
A \textit{fundamental domain} $F$ of $H$ in $G\subseteq \RR^d$ is a set that contains exactly one representative from each class in $G/H$. A measurable, bounded common fundamental domain of a given full-rank lattice $L\subseteq \RR^d$ in $\RR^d$ always exists. A standard example of such a set is the so-called \textit{fundamental parallelepiped} of $L$ with respect to a given basis of $L=<v_1,...,v_d>_\ZZ$:
\begin{equation*}
    P_L := \Bigl\{\sum_{i\le d} t_iv_i \: | \: t_i \in [-\frac{1}{2}, \frac{1}{2})\Bigr\}
\end{equation*}
\begin{figure}[H]
    \centering
    \begin{subfigure}{0.45\textwidth}
        \centering
        \includegraphics[width=\linewidth]{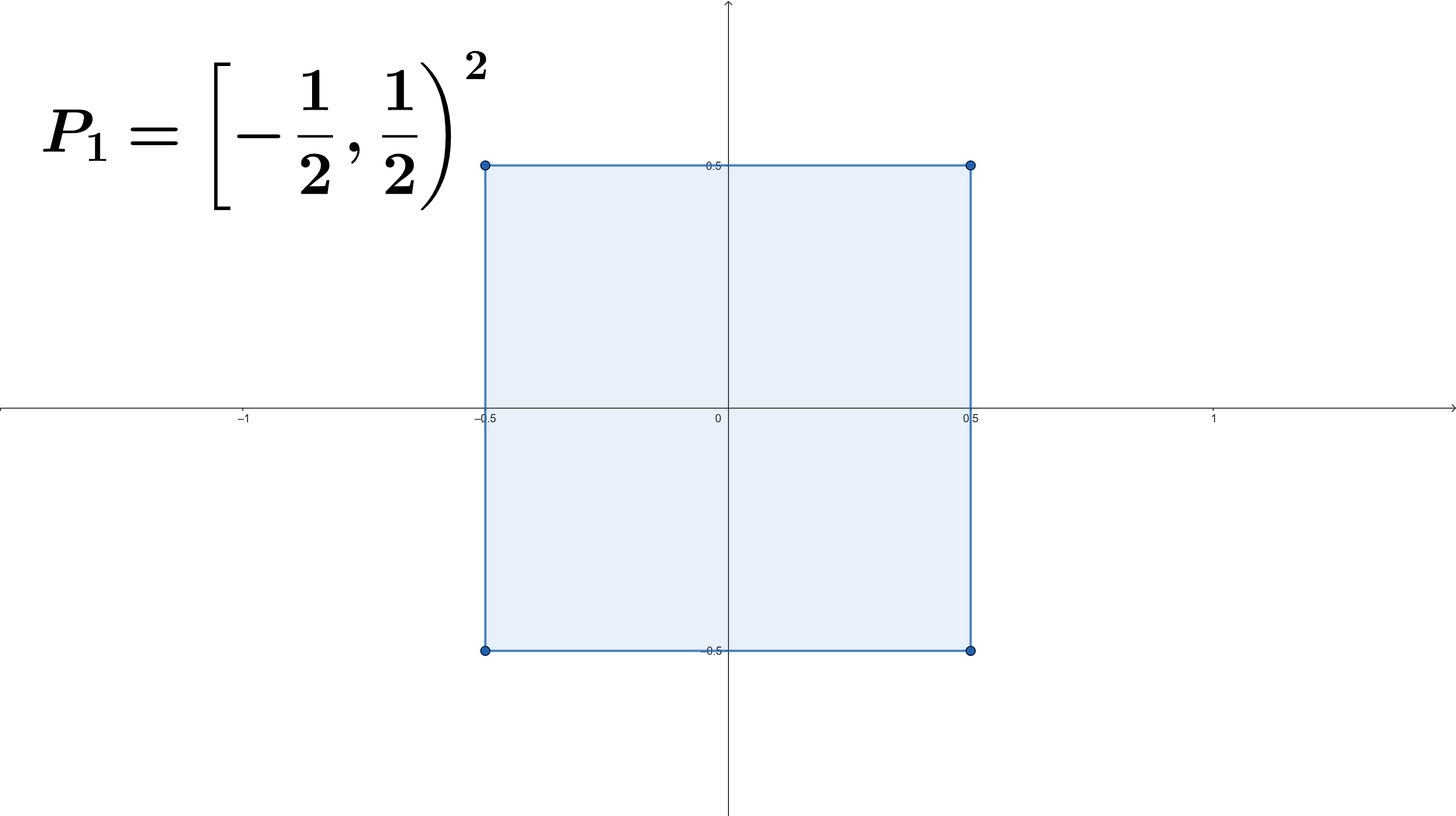}
        \label{fig:left}
    \end{subfigure}
    \hfill
    \begin{subfigure}{0.45\textwidth}
        \centering
        \includegraphics[width=\linewidth]{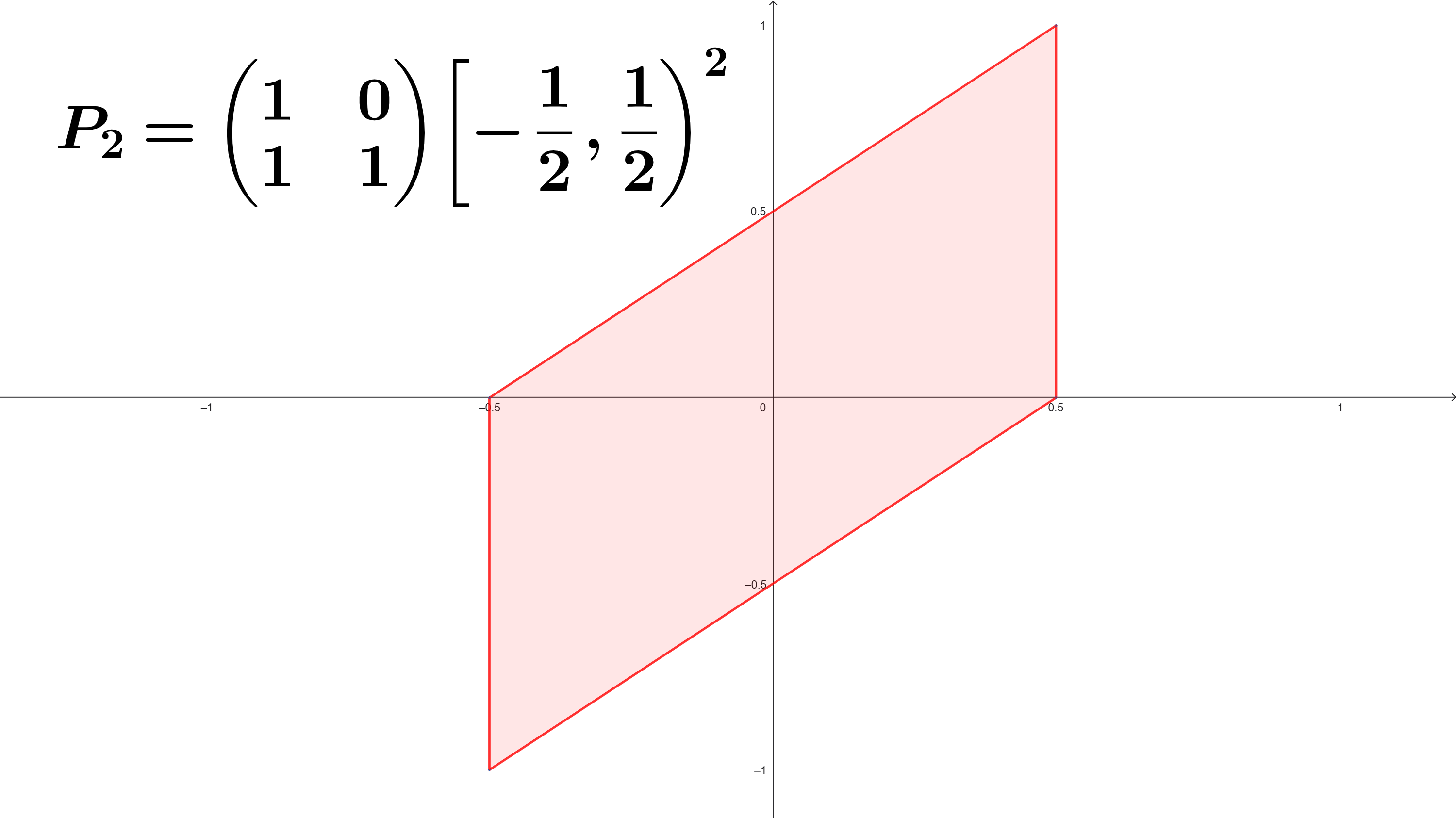}
        \label{fig:right}
    \end{subfigure}
    \caption{Two different fundamental parallelepipeds of the lattice $\ZZ^2$. $P_1$ with respect to the basis $\{(1,0), (0,1)\}$ (left) and  $P_2$ with respect to the basis $\{(1,1),(0,1)\}$ (right).}
\end{figure}

The following lemma states that from a fundamental domain of a lattice $L$ in $\RR^d$ we can always find a fundamental domain of $L$ in a subgroup $G\subseteq \RR^d$:
\begin{lemma}\label{FcapG}
Let $L\subseteq \RR^d$ be a full-rank lattice and $G$ any subgroup of $\RR^d$ such that:
\begin{equation*}
    L\subseteq G\subseteq\RR^d
\end{equation*}
If $F$ is a fundamental domain of $L$ in $\RR^d$, then $F\cap G$ is a fundamental domain of $L$ in $G$.
\end{lemma}
\begin{proof}
    We will first show that $F\cap G$ contains at least one representative from each class in the quotient group $G/L$. Consider a class $g\mod(L) \in G/L$ for some $g\in G$. Since $F$ is a  fundamental domain of $L$ in $\RR^d$ there exists an element $x\in F$ for which:
    \begin{equation*}
        x=g\mod(L)
    \end{equation*}
    watching $g\mod(L)$ as a class in $\RR^d/L$.
    This implies that $x=g+l$ for some $l\in L$. Since $G$ is a group containing $L$ and $g\in G$, we have that $x\in F\cap G$.

    Now we will show that $F\cap G$ contains at most one representative from each class in $G/L$. This comes immediately from the fact that $F$ contains exactly one element from each class in $\RR^d/L$ as a fundamental domain of $L$ in $\RR^d$, and the proof is complete.
    \end{proof}
    \subsection{Equidecomposability}
    Consider a space $X$ endowed with an action of a group $G$. We use $g\cdot x$ for the action of an element $g\in G$ acting on a point $x\in X$.

    \begin{definition}\label{equidecomposition}
    We say that two sets $A, B\subseteq X$ are $G$-equidecomposable if there exist finitely many sets $A_1,..., A_n \subseteq X$ and elements $g_1,...,g_n\in G$ such that $\{A_j\}_{j\le n}$ forms a partition of $A$ while $\{g_j \cdot A_j\}_{j\le n}$ forms a partition of $B$. The sets $A_1,...,A_n$ are called pieces of the $G$-equidecomposition between $A,B$.
    \end{definition}

    It is known and easy to see that equidecomposability with respect to a group action $G$ on a space $X$ is an equivalence relation. For this reason, whenever we have two sets $A,B\subseteq X$ that are $G$-equidecomposable, we will at times write:
    \begin{equation*}
        A\overset{G}{\sim} B
    \end{equation*}  

\textbf{Remark}.
	Throughout this text, $X$ will be a locally compact subset of $\RR^d$, usually a bounded fundamental domain of a full-rank lattice $L$ inside a subgroup $H\supseteq L$ of $\RR^d$. As a fundamental domain of $L$ in $H$, $X$ can be identified with the quotient group $H/L$. Via this identification, $X$ acquires a group structure induced by the abelian group structure of $H/L$. As for the action group $G$ of $X$, will usually be a subgroup of the quotient group $H/L$ and for this reason we denote the action of an element $g\in G$ on a point $x\in X$ by $x+g$. In other words $G$ acts on $X$ by translations.
	
    \begin{figure}[h]
    	\centering
    	\includegraphics[width=0.6\textwidth]{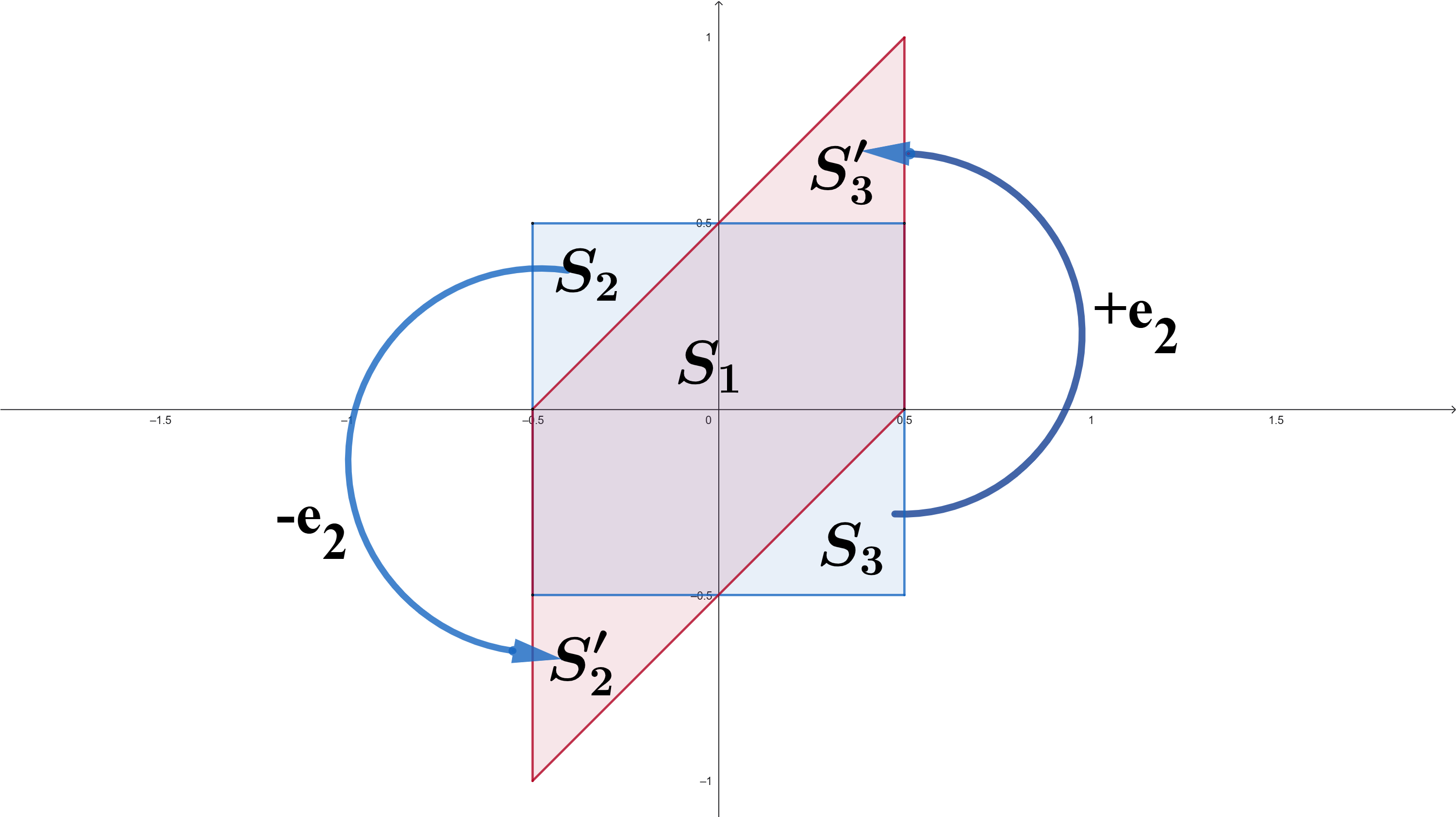}
    	\caption{An example of two $\ZZ^2$- equidecomposable sets, $P_1= [-\frac{1}{2}, \frac{1}{2})^2$ appears in blue and $P_2=\begin{pmatrix}
    			1 & 0\\
    			1 & 1
    		\end{pmatrix}\left[-\frac{1}{2}, \frac{1}{2}\right) ^2$ in red. The disjoint pieces $S_1 =P_1 \cap P_2,S_2,S_3$ partition $P_1$ and $S_1, S_2'=S_2 -e_2, S_3 '=S_3 +e_2$ partitions $P_2$ where $e_2=(0,1)\in \ZZ^2$.}
    	\label{fig:myimage}
    \end{figure}
    The following simple lemma clarifies the connection between the notion of equidecomposability and the existence of a bounded common fundamental domain for two full-rank lattices on $\RR^d$.
    \begin{lemma}\label{eq}
    Let $L, M\subseteq \RR^d$ be two full-rank lattices, and let $P_L$ and $P_M$ be fundamental parallelepipeds of $L$ and $M$, respectively.

    Then, $L$ and $M$ admit a bounded common fundamental domain on $\RR^d$ if and only if
    \begin{equation*}
        P_L\overset{L+M}{\sim}P_M
    \end{equation*}
    in $\RR^d$. Moreover, if

    \begin{equation*}
        P_L \overset{L+M}{\sim} P_M
    \end{equation*}
with Lebesgue measurable pieces in $\RR^d$ then $L,M$ admit a bounded and Lebesgue \textbf{measurable} common fundamental domain on $\RR^d$.
\end{lemma}
\begin{proof}
    We start with the proof of the “only if” direction. Let $F$ be a bounded common fundamental domain of $ L$ and $ M$. We will show that $P_L \overset{L+M}{\sim} F$ and similarly, one can obtain that $P_M\overset{L+M}{\sim}F$, for any fundamental parallelepipeds $P_L,P_M$ of $L,M$ respectively. Then, since the equidecompasability relation with respect to the group $L+M$ is an equivalence relation, we get that:
    \begin{equation*}
        P_L\overset{L+M}{\sim}P_M
    \end{equation*}
    and thus the proof of this direction will be complete. To this end, since $L,M\subseteq L+M$, it is enough for us to show that $P_L \overset{L}{\sim} F$; similarly, one can show that $P_M \overset{M}{\sim} F$. Define the set:
    \begin{equation*}
        A= \{ l\in L\: | \: x+l \in F, \: \text{for some} \: x\in P_L\}
    \end{equation*}
    and notice that since $F, P_L$ are bounded sets so is $A\subseteq L$ and hence $A$ is finite. For every $l\in A$ define the set:
    \begin{equation*}
        S_l = \{ x\in P_L \: | \: x+l \in F\}
    \end{equation*}
    and observe that since $F$ is a fundamental domain of $L$, the collection $\{S_l \}_{l\in A}$ forms a finite partition of $P_L$. On the other hand, since $P_L$ is a fundamental domain of $L$, the collection $\{S_l +l\}_{l\in A}$ forms a finite partition of $F$ and thus we obtain:
    \begin{equation*}
        P_L \overset{L}{\sim} F
    \end{equation*}
    as promised.

    As for the proof of the "if" direction, from $P_L \overset{L+M}{\sim} P_M$, we get that there exists a partition of $P_L$,
    \begin{equation}\label{P_L}
        P_L = \bigcup_{i\le N} S_i
    \end{equation}
    for some $N\in \NN$ and elements of $L+M$, $l_i +m_i$ where $l_i \in L$, $m_i \in M$ for every $i\le N$, such that
    \begin{equation*}
        P_M = \bigcup_{i\le N} S_i + l_i +m_i
    \end{equation*}
    is a partition of $P_M$.

    Define the set:
    \begin{equation*}
    F= \bigcup_{i\le N} S_i + l_i    
    \end{equation*}
    and notice that since $P_L$ is a fundamental domain  of $L$, the pieces $S_i +l_i$, for $i\le N$ are pairwise disjoint. Observe that $F$ is a bounded set as a finite union of translated bounded pieces $S_i \subseteq P_L$. $F$ is $L$ equidecomposable with $P_L$ and $M$ equidecomposable with $P_M$ and hence it is a bounded common fundamental domain of $L,M$. Finally, if the pieces $S_i$ in (\ref{P_L}) are measurable, then $F$ is also measurable.   
\end{proof}
\begin{figure}[H]
  \centering
  \begin{minipage}{0.45\textwidth}
    \centering
    \includegraphics[width=\textwidth]{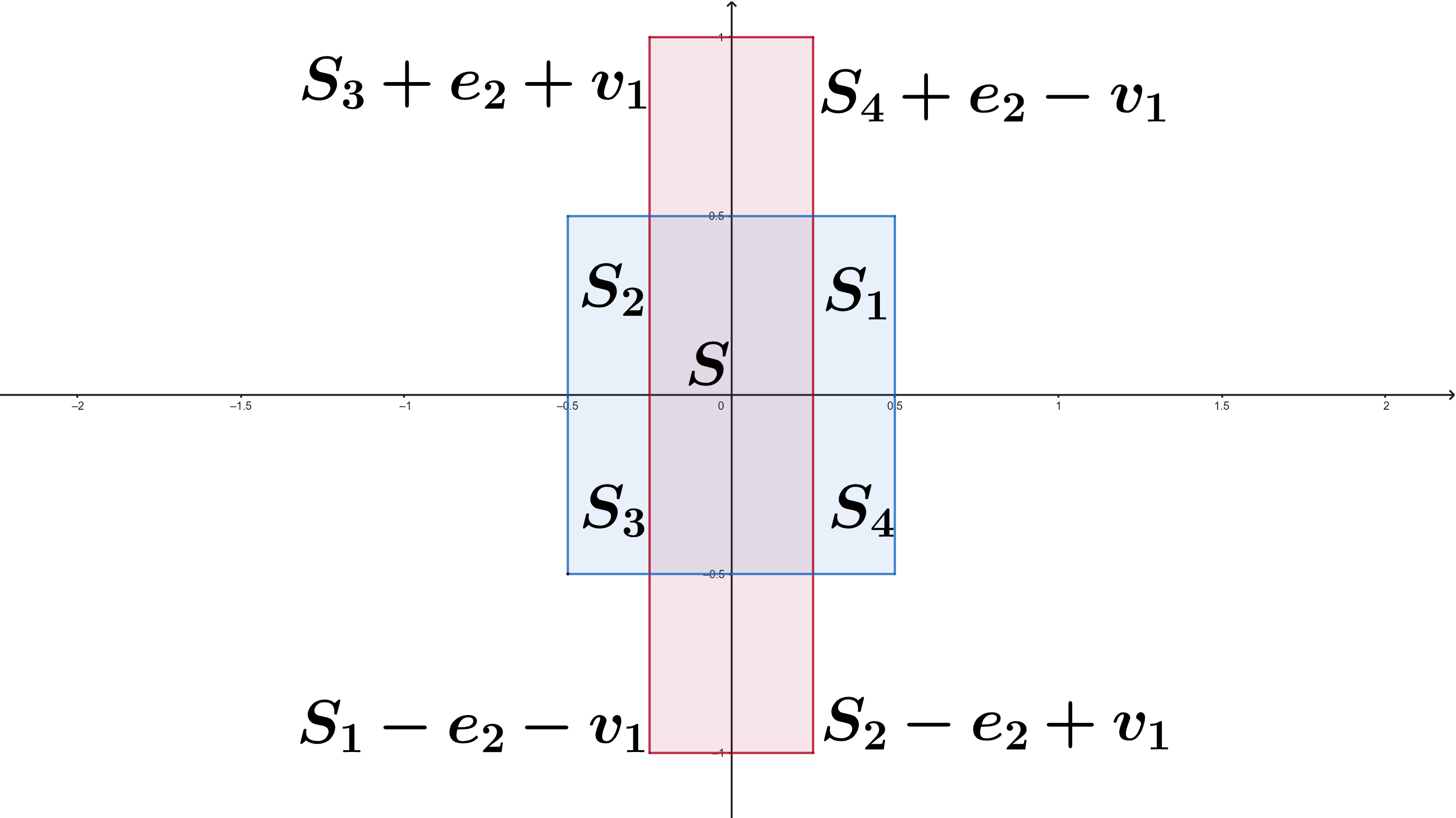}
    \caption{
    $L+M$ equidecomposition of $P_L$ (in blue), $P_M$ (in red) when $L=\ZZ^2$ and
    $M= <(1/2,0), (0,2)>_\ZZ$.}
    \label{PLPM}
  \end{minipage}
  \hfill
  \begin{minipage}{0.45\textwidth}
    \centering
    \includegraphics[width=\textwidth]{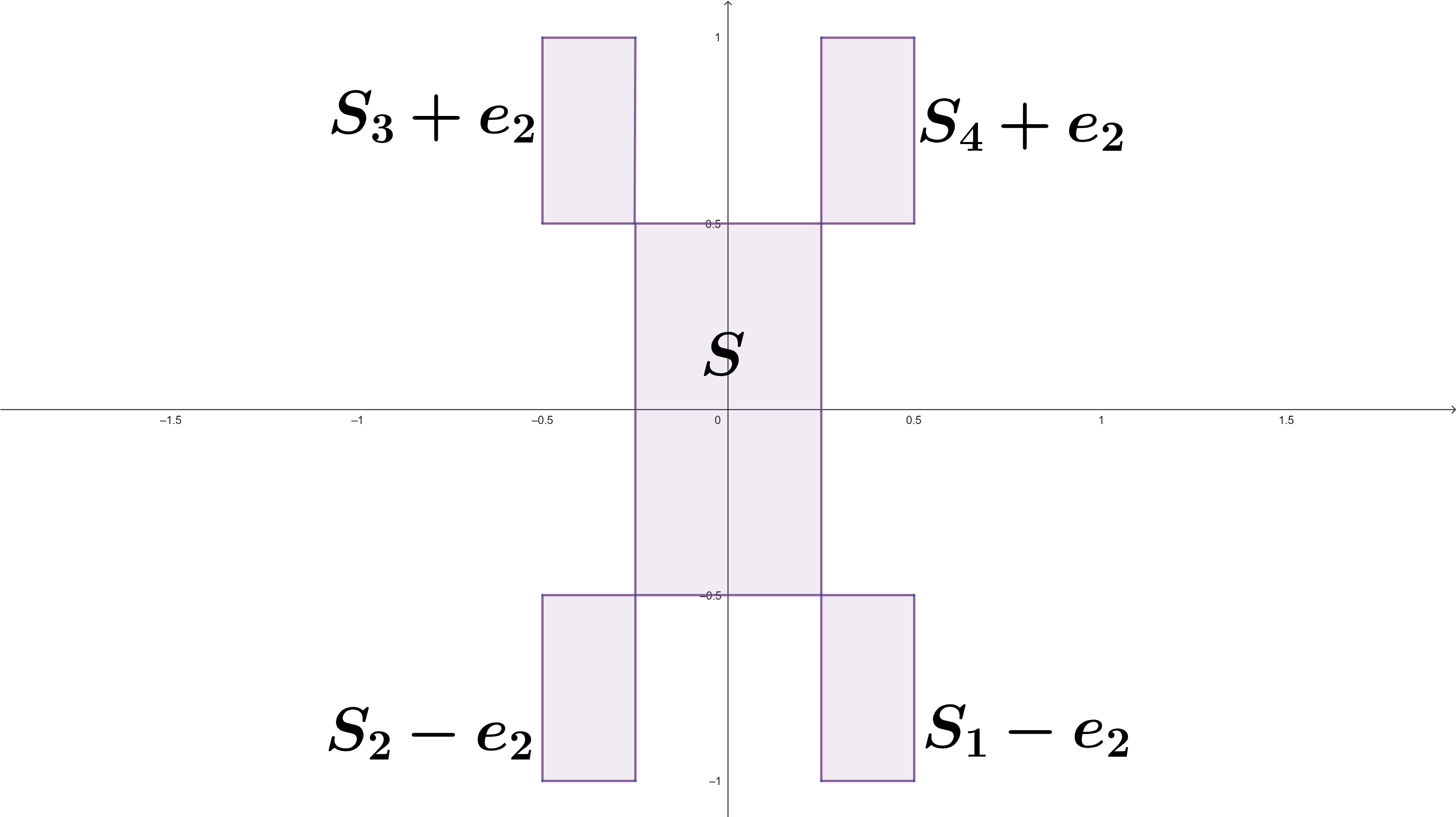}
    \caption{The bounded common fundamental domain $F$ of $L$,$M$ in Figure \ref{PLPM}.}
    \label{F}
  \end{minipage}
\end{figure} 
\subsection{Equidecomposability up to measure zero and almost-fundamental domains}
We now give another definition of equidecomposition when the space $X$ is also a measurable space.
\begin{definition}[Equidecomposition up to measure zero]\label{eq0}
    Let $(X,\mu)$ be a measurable space, either finite or infinite, endowed with a measure preserving action of a \textit{countable} group $G$. We say that two measurable sets  $A, B\subseteq X$ are $G$-equidecomposable up to measure zero, if there exist finitely many sets $A_ 1,.. ,A_n\subseteq X$, elements $g_1,..,g_n\in G$ and a full measure subset $X'\subseteq X$, such that $\{A_j \cap X'\}_{j\le n}$ forms a partition of $A\cap X'$ while $\{(A_j +g_j)\cap X'\}_{j\le n}$ forms a partition of $B\cap X'$. If the sets $A_1,..., A_n$ can be chosen measurable, then we say that $A, B$ are $G$-almost equidecomposable. 
\end{definition}

It is clear that if two measurable sets $A, B$ are $G$-equidecomposable with the Definition \ref{equidecomposition}, then they are $G$- equidecomposable up to measure zero. Notice also that equidecomposability up to measure zero, with respect to an action $G$ on a measurable space $X$, is an equivalence relation.

One can verify that any two measurable and bounded fundamental domains of a full-rank lattice $L$ are $L$-equidecomposable up to measure zero in $\RR^d$ as we showed in Lemma \ref{eq}. However, the converse is not true. For example, one may see that if $P_L$ is a fundamental parallelepiped of $L$ and $\overline{P_L}$ is the closure of $P_L$, then $P_L,\overline{P_L}$ are $L$-equidecomposable up to (Lebesgue) measure zero in $\RR^d$, since the boundary of $P_L$ is a null set, but $\overline{P_L}$ is not a fundamental domain of $L$.

Even though sets that are $L$- almost equidecomposable, with a measurable fundamental domain of $L$, are not in general fundamental domains of $L$, they tile $\RR^d$ by translations of $L$. In other words if $F$ is $L$-almost equidecomposable with a fundamental domain $F'$ of $L$, then:
\begin{equation}\label{tiling}
    \sum_{l\in L} 1_{F} (x-l)=1
\end{equation}
for almost every $x\in \RR^d$. To see this, assume that $F'\cap X= \bigcup_{i\le n} S_i\cap X$ such that $F\cap X= \bigcup_{i\le n} (S_i+l_i)\cap X$ for some full measure subset $X\subseteq \RR^d$, measurable sets $S_i$ and elements of $L$, $l_i$ for $i\le n$, where $n\in \NN$. Since $X$ is a full measure set on $\RR^d$ we have:  $1_{F'}= 1_{\bigcup_{i\le n} S_i}$ and $1_{F} = 1_{\bigcup{i\le n} \: \: \: S_i+l_i}$, $m_d$- a.e. on $\RR^d$, where $1_{F}$ is the indicator function of $F$ and $m_d$ the Lebesgue measure on $\RR^d$. Furthermore, since $S_i\cap X$ for $i\le n$, are pairwise disjoint and $X$ is a full measure set in the sense that $m_d(\RR^d\setminus X)=0$, it follows that $S_i $ for $i\le n$ are pairwise disjoint in measure (i.e., $m_d(S_i\cap S_j)=0$ for $i\neq j$) and so $1_{F'}= \sum_{i\le n} 1_{S_i}$, a.e. on $\RR^d$. Similarly, we have that $1_{F}= \sum_{i\le n} 1_{S_i +l_i} $ a.e on $\RR^d$. Finally since $F'$ is a measurable fundamental domain of $L$ and $L$ is a countable subgroup of $\RR^d$, we have for almost every $x\in \RR^d$:
\begin{equation*}
    1=\sum_{l\in L} 1_{F'} (x-l)=\sum_{l\in L} \sum_{i\le n} 1_{S_i}(x-l)=\sum_{l\in L}\sum_{i\le n} 1_{S_i +l_i} (x-(l+l_i))= \sum_{i\le n}\sum_{l\in L}1_{S_i+l_i}(x-l)=\sum_{l\in L}1_F (x-l) 
\end{equation*}

Measurable sets that tile $\RR^d$ by translations of a full-rank lattice $L\subseteq \RR^d$ are called \textit{almost-fundamental domains} of $L$. An elaboration of the arguments used in Lemma \ref{eq} gives us a sufficient condition for the existence of a bounded common, almost-fundamental domain of two full-rank lattices:
\begin{lemma}\label{Afd}
    Let $L, M\subseteq \RR^d$ be two full-rank lattices, and let $P_L$ and $P_M$ be fundamental paralllelepipeds of $L$ and $M$, respectively.

    If $P_L, P _M$ are $L+M$ equidecomposable up to measure zero, with measurable pieces on $\RR^d$ then there exists a bounded common almost-fundamental domain of $L, M$ in $\RR^d$.
\end{lemma}
\begin{figure}[h]
	\centering
	\begin{minipage}{0.45\textwidth}
		\centering
		\includegraphics[width=\textwidth]{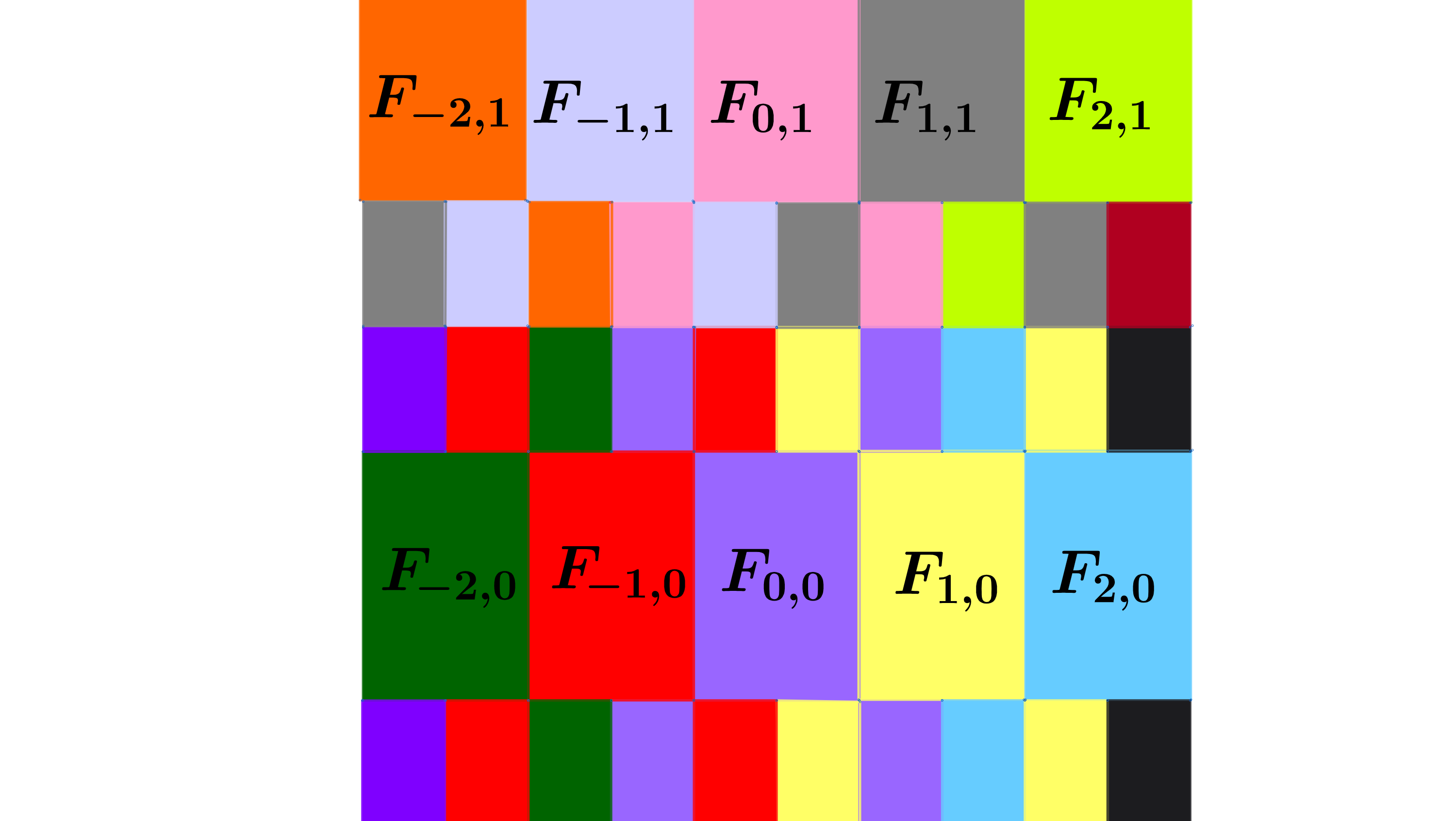}
	\end{minipage}
	\begin{minipage}{0.45\textwidth}
		\centering
		\includegraphics[width=\textwidth]{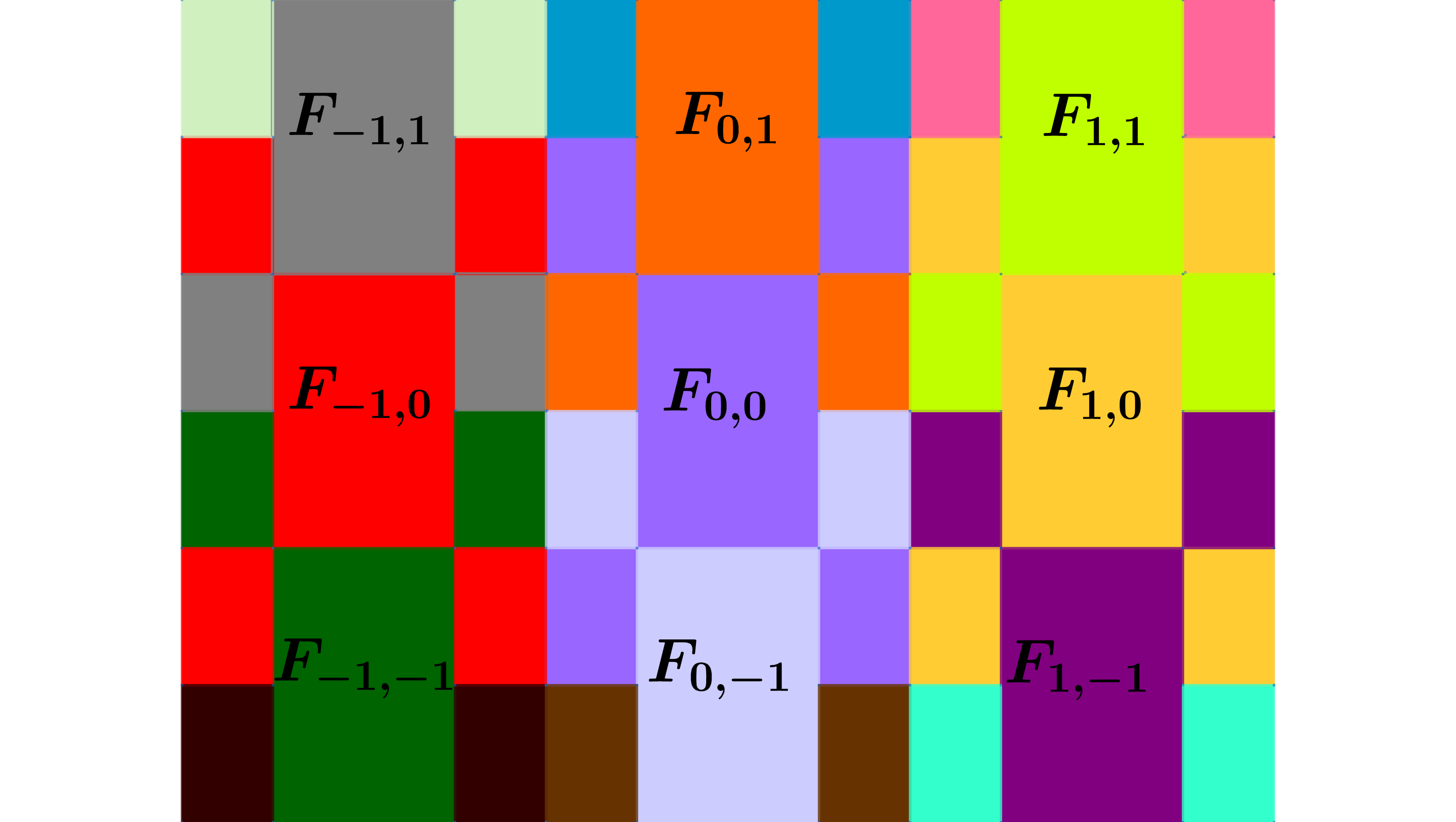}
	\end{minipage}
	\caption{The set $F$ in Figure \ref{F} tiles $\RR^2$ by translations of $M= <(1/2,0),(0,2)>_\ZZ$ (left) and by translations of $L=\ZZ^2$ (right), as a common fundamental domain of $L,M$. In the figure to the left $F_{n,m}=F + n(1/2, 0)+ m(0,2)$ and to the right $F_{n,m}= F + n(1,0)+ m(0,1)$, for $n,m\in \ZZ$.}
\end{figure}

Finally, we state the following result in \cite{Kol97}, which gives us a sufficient condition for the existence of a measurable common fundamental domain of two full-rank lattices $L,  M$, if we have a common almost-fundamental domain of $ L, M$:
\begin{theorem}[{\cite[Theorem 3]{Kol97}}]\label{MeasurableFD}
    If two full-rank lattices $L, M\subseteq\RR^d$ have a not necessarily measurable, common fundamental domain $\Omega'$ and a measurable common almost-fundamental domain $\Omega''$, then they have a Lebesgue measurable common fundamental domain $\Omega\subseteq \Omega'\cup \Omega''$ in $\RR^d$.

    If $\Omega',\Omega''$ are bounded sets, then $\Omega$ is also a bounded set.
\end{theorem}
\subsection{The measurable Hall's Theorem}
Here we state the measurable Hall's Theorem proved by T.Cieśla and M.Sabok \cite{CS22}, which gives conditions guaranteeing equidecomposition of two measurable sets with measurable pieces.

Assume $(X,\mu)$ is a standard Borel probability space, endowed with a free pmp (probability measure preserving) action of a finitely generated group $G$. We recall that the action of $G$ on $X$ is called free if $g+x\neq x$ for every nontrivial
element $g\in G$ and every $x\in X$.

We remind to the reader that the rank of a finitely generated abelian group $G$ is the minimal number of generators that produce $G$. By the structure theorem for finitely generated abelian groups, we may assume
that $G=\ZZ^d \times \Delta$ where $d$ is the rank of $G$ and $\Delta$ is a finite abelian group. 

\begin{definition}[see \cite{CS22}, Definition 5]
A measurable set $A\subseteq X$ is called $G$-uniform if there exist positive constants $c$ and $n_0$, such that for almost every $x\in X$
and for every $n > n_0$ we have $|A\cap(F_n + x)|\ge cn^d$
, where $F_n := \{0,1,...,n-1\}^d\times \Delta$.
\end{definition}
The following lemma gives us a sufficient condition  for a measurable set to be $G$- uniform:

\begin{lemma}\label{Guni}
    Let $X\subseteq \RR^d$ be a locally compact abelian group and let $\mu$ be a Borel probability measure on $X$ that is translation-invariant (i.e., $\mu$ is the normalized haar measure on $X$). Let $G\subseteq X$ be a finitely generated dense subgroup of $X$ of rank $r>0$, acting on $X$ by translations.

    Then any Borel measurable set $A\subseteq X$ with nonempty interior is a $G$-uniform set in $X$.
\end{lemma}
\begin{proof}
By the structure theorem for finitely generated abelian groups, there exists a
direct sum decomposition $G= M \oplus \Delta$ where $M$ is a free abelian group of positive rank $r$ and $\Delta$ is a finite group. Since $G$ is dense in $X$, then it follows that $M$ is also dense in $X$ (see \cite{Ru62} Section 2.1). Let $m_1,..,m_r$ be a set of generators of $M$ and denote for every $k\in \NN$:
\begin{equation*}
    F_k = P_k \oplus \Delta, \: \: \: P_k =\left\{\sum_{i\le r} n_im_i \: | \: n_i \in \{ 0,...,k-1\}, i\le r \right\}
\end{equation*}
    For a measurable set $A\subseteq X$ to be $G$ uniform on $X$ it suffices to show the existence of a positive constant $c>0$ and $k_0\in \NN$ such that for every $k\ge k_0$ and every $x\in X$:
    \begin{equation*}
        |A\cap F_k +x|\ge c k^r
    \end{equation*}

    Since $A$ has nonempty interior, there exists a ball $B$ of radius $2\epsilon$ for some $\epsilon>0$, contained in $A$. Since $M$ is dense in $X$, there is a positive integer $k_0$ such that the set $P_{k_0}$ forms an $\epsilon$-net in $X$. This implies that any translation of $P_{k_0}$ is also an $\epsilon$-net in $X$. Now observe that for every $x \in X$ and every $k > k_0$, the set $P_k + x$ contains at least $\Floor{k/k_0}^r$ disjoint translated copies of $P_{k_0}$, each of which must intersect the ball $B$. Hence, for every $x\in X$ and every $k> k_0$:
    \begin{equation*}
        |A\cap F_k +x| \ge |B\cap P_k +x| \ge \Floor{k/k_0}^r\ge ck^r
    \end{equation*}
    so $A$ is a $G$-uniform set on $X$.
\end{proof}
We now state the measurable Hall's theorem of T. Cieśla  and M. Sabok:
\begin{theorem}[{\cite[Theorem 2]{CS22}}]\label{CS22}
Let $(X, \mu)$ be a standard Borel probability space,
endowed with a free pmp action of a finitely generated abelian group $G$, and let $A, B\subseteq X$ be two measurable $G$-uniform sets. Then the following conditions are equivalent:

(a) $A$ and $B$ satisfy Hall’s condition a.e. with respect to $G$;

(b)$A$ and $B$ are $G$-equidecomposable up to measure zero (with possibly non-measurable pieces);

(c) $A$ and $B$ are $G$- almost equidecomposable.
\end{theorem}

\section{Bounded common fundamental domains}
Here we give the proof of Theorem \ref{BF}.
\begin{proof}[Proof of Theorem \ref{BF}]
    Consider two full-rank lattices $L, M$ with equal volumes. Since every closed subgroup of $\RR^d$ is up to a non-singular linear transformation, a set of the form:
    \begin{equation}\label{L+M}
        \ZZ^m \times \RR^{d-m}
    \end{equation}
    where $m\in \{0,1,..,d\}$ (see \cite{HR12}, Theorem 9.11) we may assume that $\overline{L+M}$ is a set of the form (\ref{L+M}) for some $m\in \{0,..d\}$.

    We will show that it is possible to construct a common bounded fundamental domain $F$ for $L, M$ in $\RR^d$, by first constructing a bounded common fundamental domain $F_1$ of $L, M$ inside the group $L+M$. Assuming this is done, it is easy to find a bounded fundamental domain $F_2$ of $L+M$ inside $\overline{L+M}=\ZZ^m \times \RR^{d-m}$. Indeed, since $L+M$ is dense in $\overline{L+M}$, we may choose a representative from each class in $\overline{L+M}/(L+M)$  that lies inside the ball centered at the origin of some radius $\epsilon>0$. $F_2$ will then be the set that contains one such representative from each class in $\overline{L+M}/(L+M)$, which is a bounded fundamental domain of $L+M$ inside $\overline{L+M}=\ZZ^m \times \RR^{d-m}$. Finally, the desired common fundamental domain of $L,M$ inside $\RR^d$ will then be:
    \begin{equation*}
        F= [0,1)^m \times \{0\}^{d-m} + F_1 +F_2
    \end{equation*}
    which is, from the way it was constructed, a bounded set.
    \begin{figure}[H]
    \centering
   \begin{minipage}[b]{0.45\textwidth}
        \centering
        \includegraphics[width=\linewidth]{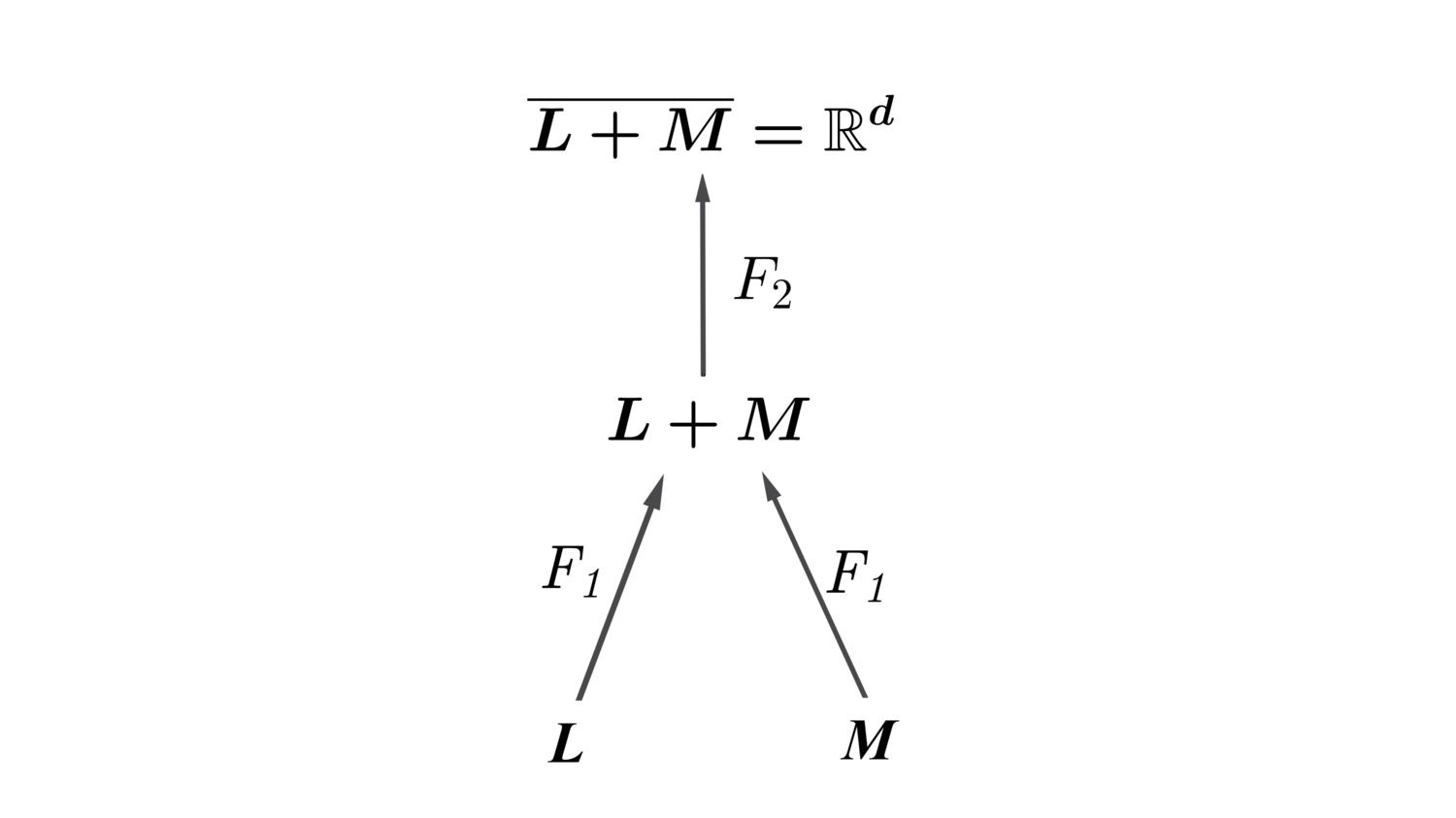}
        \caption{$F=F_1 +F_2$ when $m=0$.}
    \end{minipage}\hfill
    \begin{minipage}[b]{0.45\textwidth}
        \centering
        \includegraphics[width=\linewidth]{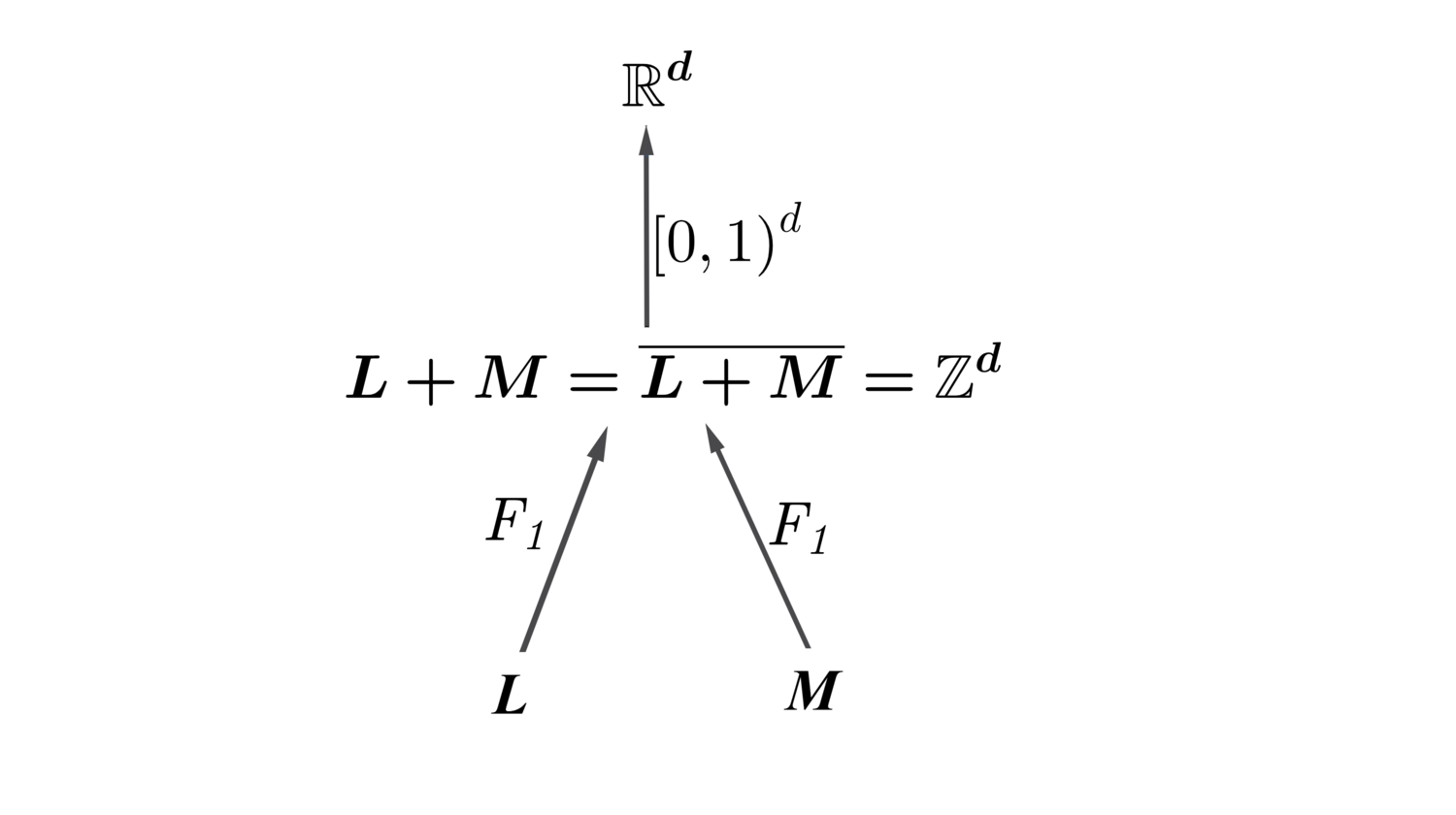}
        \caption{$F=[0,1)^d + F_1$ when $m=d$.}
        
    \end{minipage}
\end{figure}
\begin{figure}[H]
     \begin{minipage}[b]{0.6\textwidth}
		\centering
		\includegraphics[width=\linewidth]{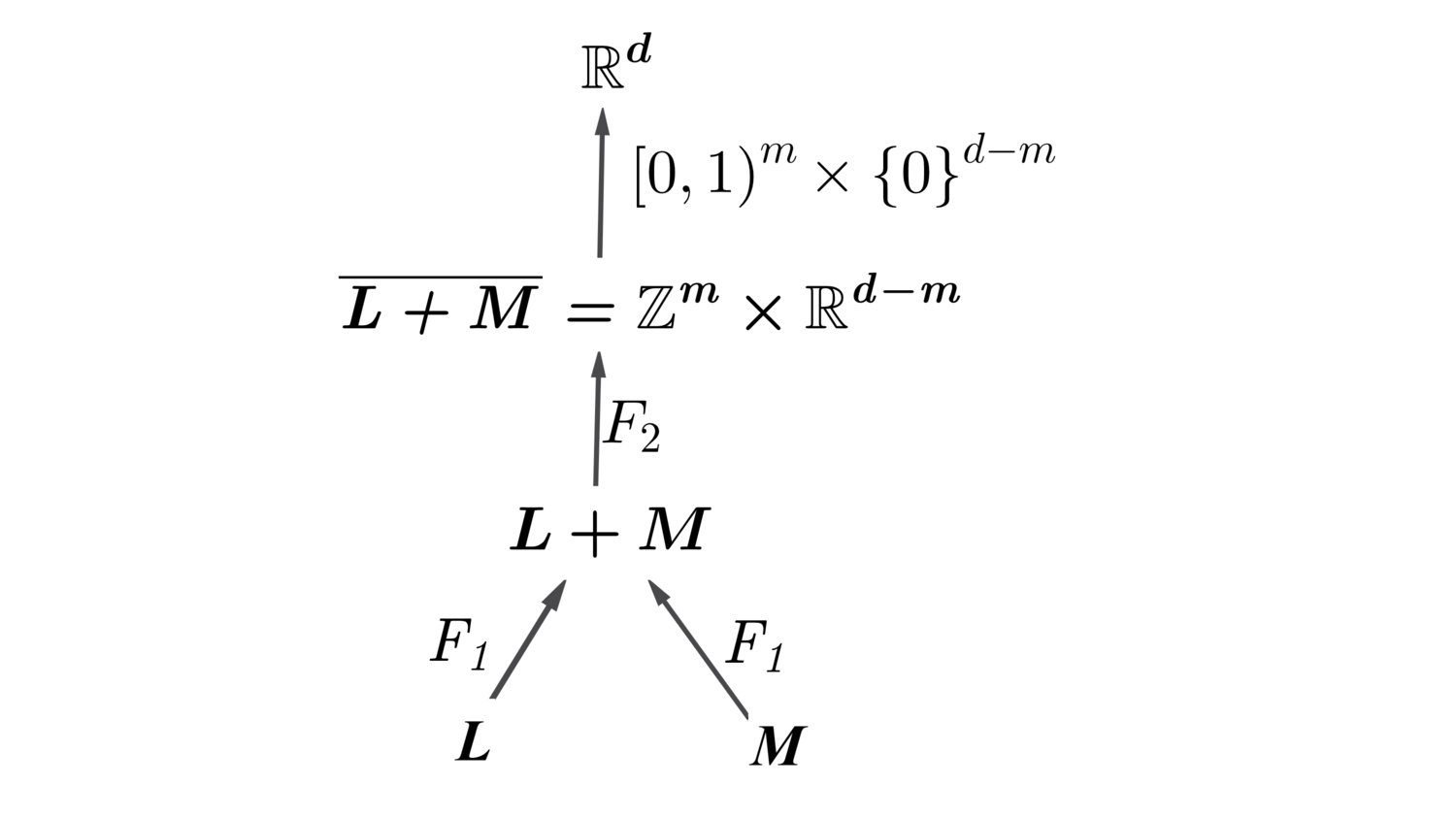}
		\caption{$F=[0,1)^{m}\times \{0\}^{d-m} + F_1 +F_2$ when $m\in \{1,...,d-1\}$}
	\end{minipage}
\end{figure}


    To this end, we are left to show the existence of a bounded common fundamental domain $F_1$ of $L, M$ inside $L+M$. For this, we  treat separately the cases depending on $L\cap M=H$, where $H$ is a lattice in $\RR^d$ of $\textit{rank}(H)=r$, when $r=0$, $r=d$ and $r\in \{1,2,...,d-1\}$
    \begin{itemize}
        \item \textbf{$1$st Case} ($r=0$) $L\cap M=\{0\}$:
        
        This case has already been proved in \cite{Kol97}
        (section $3.2$) and in fact it is an immediate application of a theorem proved by Michel Duneau and Christophe Oguey, which states that for any two lattices $L, M $ of equal volume, there exists a bijection $\phi:L \to M $ and a constant $C>0$ such that:
        \begin{equation}\label{bijection}
            ||\phi(l)-l||\le C
        \end{equation}
        for every $l\in L$ (see \cite{DO91}, Theorem $1$).

        Indeed, notice that for any enumeration of points of $L$, $l_i\in L$ for $i\in \NN$ and for any enumeration of points of $M$, $m_i\in M$ for $i \in \NN$ the set:
        \begin{equation*}
            P=\{l_i +m_i \: | \: i \in \NN\}
        \end{equation*}
        is a common fundamental domain of $L,M$ inside $L+M$ since $L\cap M=\{0\}$. Furthermore, since $L,M$ have the same volumes there exists a bijection  $\phi:L\to M$ as in (\ref{bijection}) which implies that the set:
        \begin{equation*}
            F_1= \{ l-\phi(l) \: | \: l\in L\}
        \end{equation*}
        is a common bounded fundamental domain of $L,M$ inside $L+M$.

        \item \textbf{$2$nd Case} $r=\rank(H)=d$:
        
        Since $H$ is a full-rank lattice in $\RR^d$ and $L,M$ are also full-rank lattices satisfying $\vol(L)=\vol(M)$, it follows that
        \begin{equation}\label{[L:H]}
            [L:H]=\frac{\vol(H)}{\vol(L)}=\frac{\vol(H)}{\vol(M)}=[M:H]=k
        \end{equation}
        for some $k\in \NN$. Consider the quotient groups $L/H$ and $M/H$ and notice from (\ref{[L:H]}) that they have the same order (i.e., the same number of elements). Choose one representative from each class in $L/H$, call it $l_i$ for $i\le k,$ and one representative from each class in $M/H$, call it $m_i$ for $i\le k$. The set
        \begin{equation*}
            F_1 = \{ l_i +m_i \: | \: i\le k\}
        \end{equation*}
        is a common fundamental domain of $L,M$ inside $L+M$. We will show that $F_1$ is a fundamental domain of $L$ inside $L+M$, and similarly, one can prove that it is also a fundamental domain of $M$ inside $L+M$. Indeed, $F_1$ contains at most one representative from each class in $(L+M)/L$ for, assuming otherwise we have that
        \begin{equation*}
        m_i +l_i = m_{i'} + l_{i'} \bmod(L)    
        \end{equation*}
        for some $i\neq i'$. Then  $m_i -m_{i'} \in L\cap M=H$ and so $m_i =m_{i'} \bmod(H)$, a contradiction. Moreover, $F_1$ contains at least one representative from each class in $(L+M)/L$, since if  $m\bmod(L)$ is a class in $(L+M)/L$ for some $m\in M$ then $m= m_i \bmod (H)$ for some $i\le k$, which implies that $m=m_i +h$ for some $h\in H\subseteq L$. This shows that:
        \begin{equation*}
            m\bmod(L) = m_i +h \bmod(L) = m_i \bmod(L) 
        \end{equation*}
        and so $F_1$ is a fundamental domain of $L$ and similarily of $M$ inside $L+M$. Finally, $F_1$ is bounded as a finite set.

        \item \textbf{$3$rd Case} $r=\rank(H)\in \{1,2,..d-1\}$:

        Now we move on to the last case and let $r:=\rank(H) \in \{1,2,...,d-1\}$. Choose a basis $\{h_1,..,h_r\}$ for $H= L\cap M$. Up to a non singular linear transformation we can make the following assumptions:
        \begin{itemize}[label=$\blacktriangleright$]
        	\item For every $i\le r$, $h_i$ is the $d-r+i$-th standard basis vector (i.e., $1$ in position $d-r+i$, and $0$ elsewhere) and so $H=\{0\}^{d-r}\times \ZZ^r$.
        	\item $L=\ZZ^d$
        \end{itemize}
        With these assumptions $M=\mathcal{M}\ZZ^d$ for some matrix $\mathcal{M}\in GL_d (\RR)$ with $|\det(\mathcal{M})|=\vol(M)=\vol(L)=1$. We point out that under the above assumptions, $\overline{L+M}$ may not satisfy ($\ref{L+M}$) but it won't be needed here; once we find the set $F_1$ under the new assumptions we just made, we can go back through the image of $F_1$ via the inverse of the linear transformation we used and continue the construction of the set $F$ as described at the start of the proof.
        
         Now, extend the basis $\{h_1,...,h_r\}$ of $H$,  to a basis of $L$ and to a basis of $M$, so we may get two lattices $L_1, M_1$ s.t.:

        \begin{align}\label{LM}
         L &= L_1 \oplus H  &  M &= M_1\oplus H
        \end{align} 
        (see \cite{Cas96}, Corollary 3, p. 14). Since  $H=\{0\}^{d-r}\times \ZZ^{r}$, $L=\ZZ^d$ we can assume that $L_1= \ZZ^{d-r} \times \{0\}^{r}$. Choosing appropriately a basis matrix  for $L,M$, we can write:
        \begin{align}\label{L,M}
        L &= I_d\ZZ^d & M &= \begin{pmatrix}
        B & 0\\
        C & I_r
       \end{pmatrix} \ZZ^d
       \end{align}
        Where $I_d$ is the identity $d\times d$ matrix, $I_r$ the identity $r\times r$ matrix, $B$ a $d-r \times d-r$ matrix which corresponds to the first $d-r$ coordinates of a basis of $M_1$, and $C$ is a $r\times d-r$ matrix corresponding to the last $r$ coordinates of the same basis of $M_1$. 

        Using the determinant formula for the lower triangular matrix that represents $M$ in (\ref{L,M}), along with the fact that $\vol(L)=\vol(M)$, we obtain:
        \begin{equation*}
        |\det(B)|=1
        \end{equation*}
        and so $B\ZZ^{d-r}$ forms a full-rank lattice in $\RR^{d-r}$ with volume $1$.

        In this form, for any enumeration of points of $L_1$, $l_i$, with $i\in \NN$, any enumeration of points of $M_1$, $m_i$, with $i\in \NN$ and \textbf{\textit{any points}} of $H=L\cap M$, $k_i$ with $i\in \NN$ the set:
        \begin{equation}\label{P}
            P=\{ l_i + m_i +k_i \: | \: i\in \NN\} 
        \end{equation}
        is a common fundamental domain of $L,M$ in $L+M$. We will show that $P$ is a fundamental domain of $L$ and similarily one can show that it is also a fundamental domain of $M$ in $L+M$. To this end we will first show that $P$ contains at least one element from each class in $(L+M)/L$. Consider a class in $(L+M)/L$, $m\bmod(L)$ for some $m\in M=M_1 \oplus H$ due to (\ref{LM}). So there exist $i\in \NN$ and an element $h\in H\subseteq L,M$ such that $m=m_i +h$. Observe now that
        \begin{equation*}
        	m\bmod(L)= m_i +h \bmod(L)= m_i \bmod(L)=m_i +l_i +k_i \bmod(L)
        \end{equation*}
        since $l_i +k_i \in L$ and thus $P$ contains at least one element from each class in $(L+M)/L$.
        Now let us show that $P$ contains at most one element from each class in $(L+M)/L$. Consider two indices $i,i'\in \NN$ such that
        \begin{equation*}
        	m_i +l_i +k_i = m_{i'} +l_{i'}+ k_i \bmod(L)
        \end{equation*}
        It follows that $m_i -m_i' \in L\cap M_1=\{0\}$ and so $i=i'$ and thus $P$ is a fundamental domain of $L$ and similarly of $M$ in $L+M$. Observe that the elements $k_i\in H$ may or may not all be distinct.
        
         For such a set $P$ as in (\ref{P}) to be bounded, it is sufficient for us to show that the orthogonal projection $\pi: \RR^d \to \RR^{d-r} \times \{0\}^r$ is injective on both $L_1, M_1$. First, we will show why it is sufficient and notice that $\pi(M_1)=B\ZZ^{d-r} \times \{0\}^r$ and $\pi(L_1)= \ZZ^{d-r}\times \{0\}^r=L_1$ ($\pi$ is clearly injective on $L_1$). Assuming that $\pi$ is injective on $M_1$, since $B\ZZ^{d-r}, \ZZ^{d-r}$ are full rank lattices on $\RR^{d-r}$ with equal volumes there exists a bijection $\tilde{\phi}:B\ZZ^{d-r} \to \ZZ^{d-r}$ and a constant $C>0$ s.t.:
        \begin{equation*}
            ||\tilde{\phi}(b)-b||<C
        \end{equation*}
        for all $b\in B\ZZ^{d-r}$ (see \cite{DO91}, Theorem 1). Now if $\tilde{\pi}$ is the canonical embedding from $\RR^{d-r}$ to $\RR^{d-r}\times \{0\}^{d-r}$, we can define the bijection $\phi =\tilde{\pi}\circ \tilde{\phi} \circ \tilde{\pi}^{-1} : \pi (M_1)\to \pi (L_1)=L_1$ and thus for every element $(b,0)\in \pi(M_1)= B\ZZ^{d-r} \times \{0\}^r$, with $b\in B\ZZ^{d-r}$
        \begin{equation}\label{phi}
            ||\phi(b,0)-(b,0)||= ||\tilde{\phi}(b)-b||<C
        \end{equation}
        Moreover if $\pi$ is injective on $M_1$, $\pi|_{M_1}:M_1 \to B\ZZ^{d-r} \times \{0\}^{r}$ is a bijection. Hence $ \pi^{-1}\circ\phi \circ \pi=\phi \circ \pi :M_1 \to L_1$ is a bijection. Notice that every element $m\in M_1$ can be uniquely written in the form $m=(b_m,c_m)$ where $b_m\in B\ZZ^{d-r}$, $c_m\in C\ZZ^{d-r}\subseteq \RR^r$ and choose an element  $k_m \in \{0\}^{d-r}\times \ZZ^r=H$ so that $||k_m -(0,c_m)||\le r$. It follows that
        
        \begin{align}\label{bounded}
        \|\phi \circ \pi(m) - m + k_m\|
        &= \| \phi \circ \pi(m) - (b_m,0) - (0,c_m) + k_m \| \notag \\[2mm]
        &\le \|\phi(b_m,0) - (b_m,0)\| + \| k_m - (0,c_m)\| \notag \\[2mm]
        &\le C + r
        \end{align}

        Hence, the set:
        \begin{equation*}
            F_1 := \{ \phi\circ \pi (m) -m +k_m \: | \: m\in M_1\}
        \end{equation*}
        is a common fundamental domain as a set in the form of (\ref{P}) since $\phi\circ \pi: M_1 \to L_1$ is a bijection which is also a bounded set due to (\ref{bounded}).

        \begin{figure}[H]
    \begin{minipage}[b]{1\textwidth}
        \centering
        \includegraphics[width=\linewidth]{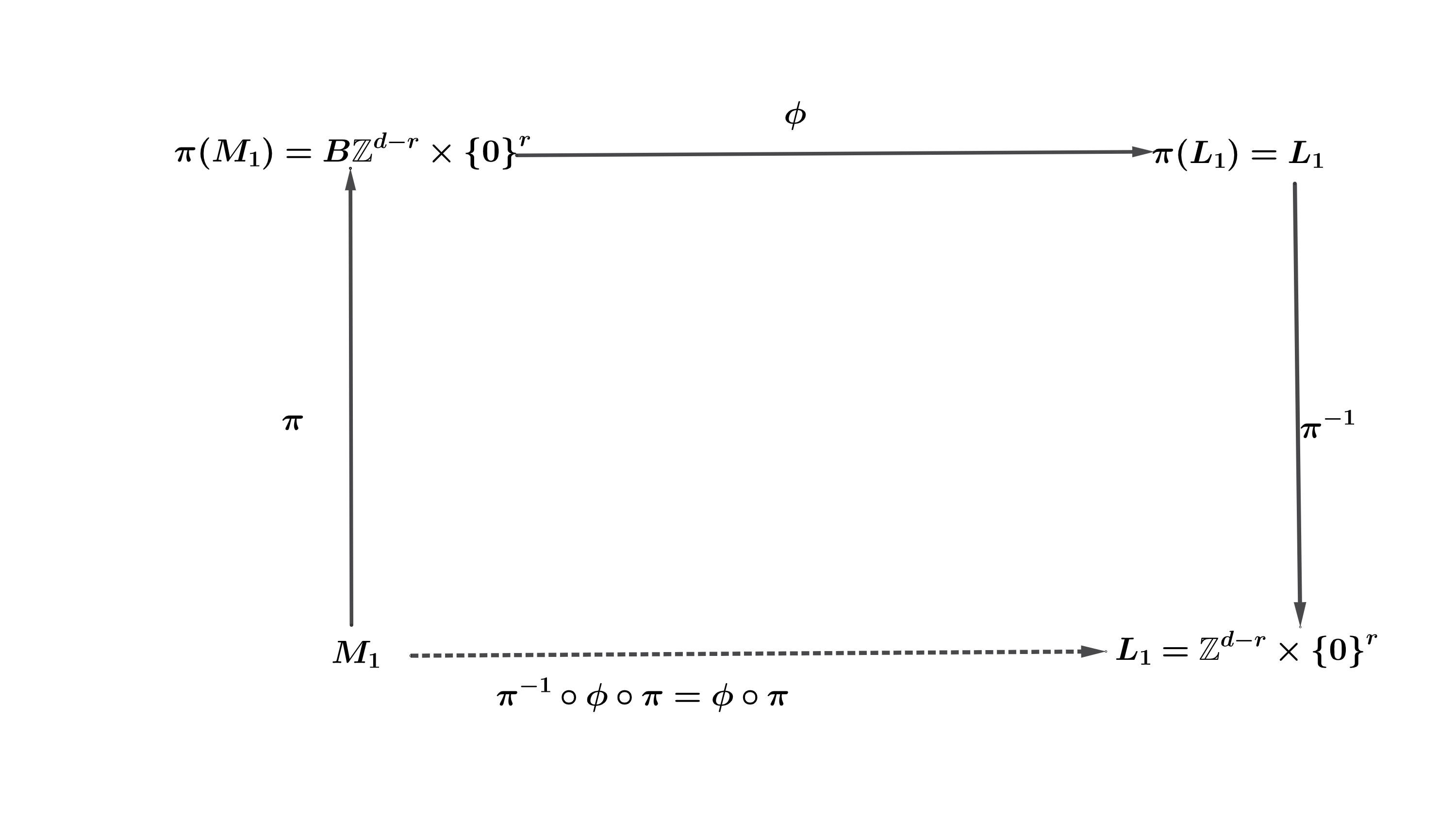}
    \end{minipage}
    \caption{The bijection $\pi^{-1}\circ\phi \circ \pi$ between $M_1, L_1$.}
    \end{figure}

        What's left for us to show is that $\pi:=\pi|_{M_1}:M_1\to B\ZZ^{d-r}\times \{0\}^r$ is injective. Notice that $M_1 = \{ (Bx,Cx)\: | \: x\in \ZZ^{d-r}\}$ due to (\ref{L,M}). Consider now two elements $x_1,x_2 \in \ZZ^{d-r}$ such that
        $\pi(Bx_1,Cx_1)= \pi(Bx_2,Cx_2)$ and so $Bx_1 = Bx_2$. Since $B$ is invertible, we obtain that $x_1=x_2$ and thus, $(Bx_1,Cx_1)= (Bx_2,Cx_2)$ and the proof is complete.        
        
    \end{itemize}
    \end{proof}

    \section{Bounded and measurable common fundamental domains}
We start this section with the proof of the Theorem \ref{BMF}.
\begin{proof}[Proof of Theorem \ref{BMF}]
    Once again, we start the proof by considering two full-rank lattices $L, M\subseteq \RR^d$ and a fundamental parallelepiped $P_L$ of $L$ and $P_M$ of $M$. We can assume that:
    \begin{equation*}
        \overline{L+M}= \ZZ^m \times\RR^{d-m}
    \end{equation*}
    for some $m\in \{0,1,...,d\}$ (see \cite{HR12}, Theorem 9.11). We will treat separately the cases: $m=0$, $m=d$ and $m\in \{1,2,...,d-1\}$. 
    \begin{itemize}
        \item \textbf{Dense on $\RR^d$ ($m=0$):}
        Consider  any system of $d$ linearly independent vectors $v_1,..,v_d\in L$ which are large enough, so that:
        \begin{equation*}
            \Omega=\left\{\sum_{i\le d } t_iv_i \:|\: t_i \in \left[-\frac{1}{2}, \frac{1}{2}\right) \right \}\supset P_L, P_M
        \end{equation*}

        Let $H$ be the full-rank lattice, spanned by $v_1,...,v_d$. In other words $H= <v_1,...,v_d>_\ZZ$.
        \begin{figure}[H]
    \begin{minipage}[b]{0.6\textwidth}
        \centering
        \includegraphics[width=\linewidth]{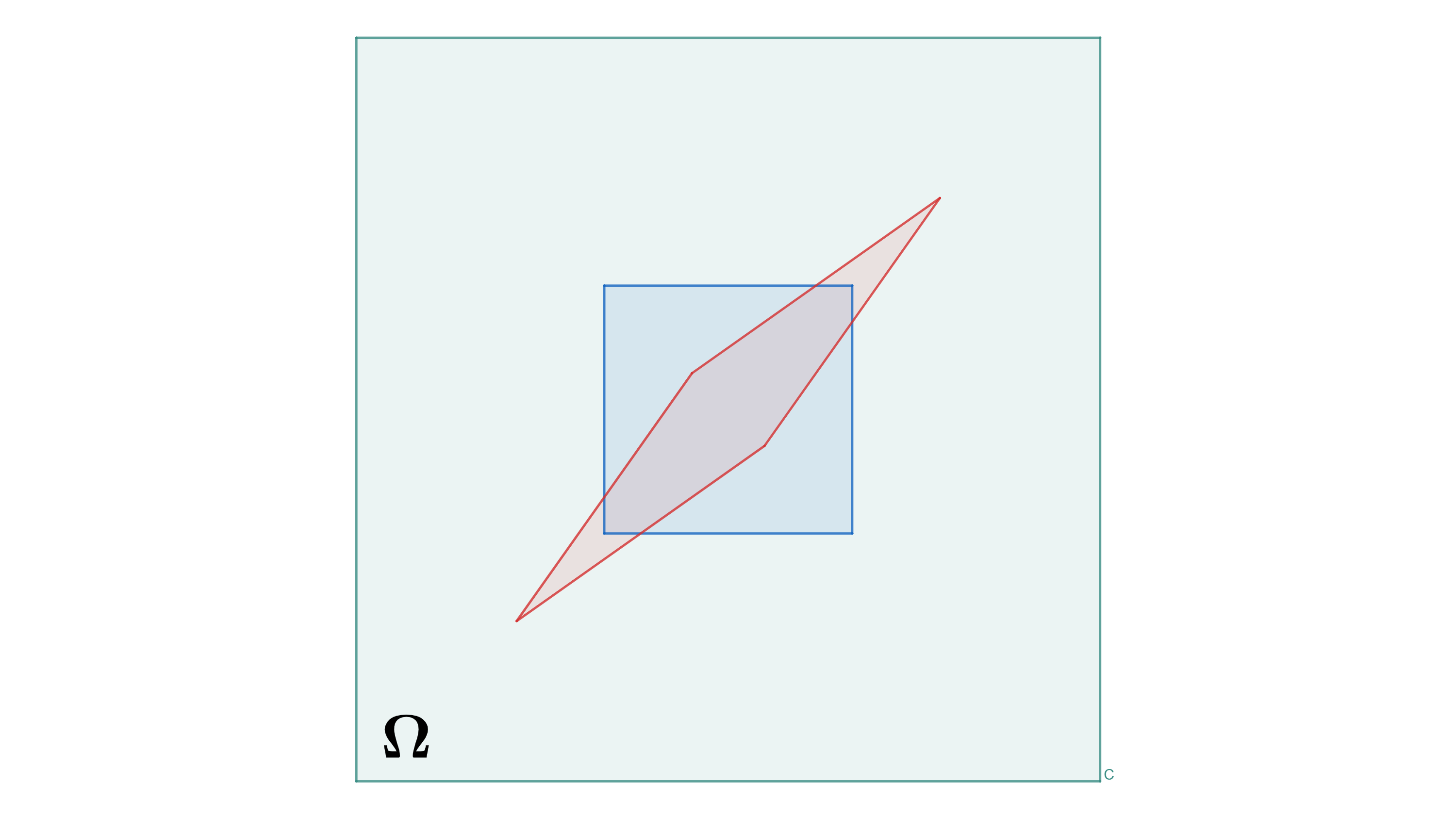}
    \end{minipage}
    \caption{The set $\Omega$ (in green) that contains $P_L=[-\frac{1}{2},\frac{1}{2})^2$ (in blue) and $P_M=\begin{pmatrix}
    		\frac{\sqrt{2}}{2} & 1 \\
    		1 & \frac{\sqrt{2}}{2}
    	\end{pmatrix}
    	\left[-\frac{1}{2}, \frac{1}{2} \right)^2$ (in red).}
    \end{figure}

        Define $X= \Omega = \RR^d/H$, seeing it as a locally compact abelian group endowed with the normalized Lebesgue measure restricted on $X$, so that $X$ is a  Borel probability space. Let $G= (L+M)/H \subseteq X$ acting on $X$ by translations. Notice that since $L+M$ is a finitely generated abelian group, so is $G$. Moreover, notice that $H$ is a full-rank sublattice of $L$. Observe that
        \begin{equation*}
        	G=(L+M)/H \supseteq (M+H)/H\cong M/(H\cap M)
        \end{equation*}
         by the second Isomorphism theorem. Notice that $\rank(M/H\cap M)>0$ for, assuming that $\rank(H\cap M)=d$ would imply that $L+M$ is a lattice and so it can't be dense in $\RR^d$. Thus, $(M+H)/H \cong M/(H\cap M)$ is infinite and so is $G\supseteq (M+H)/H$ and so $\rank(G) >0$.
        Furthermore since $L+M$ is dense on $\RR^d$, $G$ is dense on $X$. Also, $G$ has a free action on $X$ as a subgroup of $X$, and due to the probability measure on $X$ which is invariant under translations, $G$ has also a probability measure preserving action on $X$.

        By Lemma \ref{eq}, since $L,M$ admit a bounded common fundamental domain, $P_L,P_M$ are $L+M$ equidecomposable as subsets of $\RR^d$ which implies that, $P_L,P_M$ are $G$-equidecomposable as subsets of $X$. Indeed, since $P_L, P_M$ are $L+M$ equidecomposable on $\RR^d$ there exists a $N\in \NN$, pairwise disjoint subsets of $P_L$, $S_i$ for $i\le N$ such that
        \begin{equation*}
        	P_L = \bigcup_{i\le N} S_i
        \end{equation*}
        is a partition of $P_L$, and elements $l_i +m_i\in L+M$ for $i\le N$ such that
        \begin{equation*}
        	P_M = \bigcup_{i \le N} S_i + l_i +m_i
        \end{equation*}
        is a partition of $P_M$. Since $P_L,P_M \subseteq X$, it follows that $S_i \subseteq X$ are pairwise disjoint on $X$ for every $i \le N$, which obviously form a partition of $P_L$ in $X$ and also the translations of $S_i$ by $l_i +m_i \bmod(H) \in G$ for every $i\le N$, form a partition of $P_M$ in $X$.
         By Lemma \ref{Guni}, $P_L,P_M$ are $G$-uniform sets on $X$ and so we can apply Theorem \ref{CS22} to show that $P_L, P_M$ are $G$- equidecomposable with measurable pieces on $X\subseteq \RR^d$ and hence $P_L,P_M\subseteq X$ are $L+M$ equidecomposable as subsets of $\RR^d$ with measurable pieces. From Lemma \ref{Afd} we obtain a bounded common almost-fundamental domain of $L, M$ in $\RR^d$ and finally, by Theorem \ref{MeasurableFD} since $L, M$ admit a bounded common fundamental domain of $L, M$ in $\RR^d$, there exists a bounded and measurable common fundamental domain of $L, M$ inside $\RR^d$.

        \item \textbf{Fully commensurable case ($m=d$)}:

        It is known and easy to see that for any two lattices $L, M$ whose sum is $L+M=\overline{L+M}=\ZZ^d$, then $L\cap M=H$ for some full-rank lattice $H\subseteq \RR^d$. We shall prove that in this case if $L, M$ have unequal volumes, then they don't admit a bounded common fundamental domain on $\RR^d$ and so we are left to find a measurable, bounded common fundamental domain for $L, M$ whenever $\vol(L)=\vol(M)$ and $L\cap M= H$, for a full-rank lattice $H\subseteq \RR^d$. In the proof of the second case in Theorem \ref{BF}, we showed the existence of a \textbf{finite} set $F_1$, that is a common fundamental domain of $L, M$ inside $L+M= \ZZ^d$ whenever $\vol(L)=\vol(M)$. The set:
        \begin{equation*}
            F=[0,1)^d + F_1
        \end{equation*}
        is then a common bounded and measurable fundamental domain for $L, M$ on $\RR^d$.


        Let's just show that no bounded common fundamental domain of $L, M$ exists in the case where $L+M = \ZZ^d$, assuming without loss of generality that $\vol(L)<\vol(M)$. If it does exist, call $F$ such a set and apply Lemma \ref{FcapG} so that $F\cap \ZZ^d$ is a bounded (since $F$ is bounded) common fundamental domain of $L, M$ inside $L+M=\ZZ^d$. Since $\vol(L)<\vol(M)$ then $[\ZZ^d:L]=\vol(L)< \vol(M)=[\ZZ^d:M]<\infty$ which implies that the number of classes $\bmod(L)$ in $\ZZ^d/L$ is strictly less than the number of classes $\bmod(M)$ in $\ZZ^d/M$ which is also finite. This means for $F_1\cap\ZZ^d$ that $|F_1\cap \ZZ^d|=|\ZZ^d/L|=[\ZZ^d :L]$ as a fundamental domain of $L$ in $\ZZ^d$ while also  $|F_1\cap\ZZ^d|=[\ZZ^d:M]>[\ZZ^d:L]$ as a fundamental domain of $M$ in $\ZZ^d$ which is clearly a condradiction.
        \item \textbf{Intermediate case ($m\in \{1,2,...,d-1\}$}):
        Call $F_2$ a bounded common fundamental domain of $ L, M$ in $\RR^d$. By Lemma \ref{eq}, we have that $P_L,P_M$ are $L+M$ equidecomposable as subsets of $\RR^d$. Since $L+M\subset \overline{L+M}=\ZZ^m\times \RR^{d-m}$ we have that $\widetilde{P_L}:= P_L\cap \ZZ^m \times \RR^{d-m}$, $\widetilde{P_M}:=P_M\cap \ZZ^m \times \RR^{d-m}$ are $L+M$ equidecomposable as subsets of $\ZZ^m \times \RR^{d-m}$. To see this, observe that if $\{S_i\}_{i\le N}$ for some $N\in \NN$ forms a partition of $P_L$ such that the collection $\{S_i +l_i +m_i\}_{i\le N}$ forms a partition of $M$, for some $l_i +m_i \in L+M \subseteq \ZZ^m \times \RR^{d-m}$, then the collection $\{S_i \cap \ZZ^m \times \RR^{d-m}\}_{i\le N}$ forms a partition of $\widetilde{P_L}$ and also $\{(S_i \cap \ZZ^m \times \RR^{d-m}) + l_i +m_i\}_{i\le N}$ forms a partition of $\widetilde{P_M}$ since $L+M \subseteq \ZZ^m \times \RR^{d-m}$.
    
    First, we wish to obtain a $\mathcal{P}(\ZZ^m) \otimes \mathcal{L}_{d-m}$- measurable and bounded common almost-fundamental domain $F_1$ of $L,M$ in $\ZZ^m \times \RR^{d-m}$, where $\mathcal{P}(\ZZ^m)$ is the power set of $\ZZ^m$ and  $\mathcal{L}_{d-m}$ is the Lebesgue $\sigma$-algebra of subsets of $\RR^{d-m}$. We remind to the reader that
    \begin{equation*}
        \mathcal{P}(\ZZ^m)\otimes \mathcal{L}_{d-m} =\left \{ \bigcup_{n\in \ZZ^m} \{n\} \times A_n \: |\: A_n \in \mathcal{L}_{d-m} \right \}
    \end{equation*}
    In other words a set $A\in \mathcal{P}(\ZZ^m) \otimes \mathcal{L}_{d-m}$, if and only if for every $n\in \ZZ^m$ the slice of $A$:
    \begin{equation}\label{slice}
        A_n = \{ y \in \RR^n \:|\:  (n,y) \in A\} \in \mathcal{L}_{d-m}
    \end{equation}
    Assuming that such a set $F_1$ exists, the set
    \begin{equation}\label{F}
        F=[0,1)^m \times \{0\}^{d-m} + F_1= \bigcup_{n\in \ZZ^m}([0,1)^m+n )\times F_n
    \end{equation}
    where $F_n$ as in (\ref{slice}) is a $\mathcal{L}_{d-m}$ measurable, bounded  common almost-fundamental domain of $L,M$ inside $\RR^d$ and thus, our proof will be complete by Theorem \ref{MeasurableFD}. We will show that $F$ is indeed a bounded common almost-fundamental domain of $L,M$ in the end of this proof. First we will use the Cieśla-Sabok theorem (see \cite{CS22}, Theorem 2) to show the existence of such a set $F_1$.

    Again, by Lemma \ref{FcapG}, $\widetilde{P_L}:= P_L\cap \ZZ^m \times \RR^{d-m}$, $\widetilde{P_M}:=P_M\cap \ZZ^m \times \RR^{d-m}$ are fundamental domains of $L,M$ respectively in $\ZZ^m \times \RR^{d-m}\supseteq L,M$ and in fact they are both $\mathcal{P}(\ZZ^m) \otimes \mathcal{L}_{d-m}$ measurable . One way to check the measurability, is to see $\widetilde{P_L}, \widetilde{P_M}$ as the unions of the preimages of $P_L,P_M$ respectively, via the continuous maps $\phi_z : \RR^{d-m} \to \RR^d$, with:
    \begin{equation}\label{phiz}
        \phi_z (x)= (z,x)
    \end{equation} over all $z\in \ZZ^m$. Notice also that since $P_L, P_M$ are bounded sets,  $\widetilde{P_L}, \widetilde{P_M}$ are bounded, and so the number of the nonempty slices of $\widetilde{P_L}, \widetilde{P_M}$  is finite. In other words, there exists a positive integer $n$ such that for every $k\in \ZZ^m$, $|k|\ge n$ we have
    \begin{equation}\label{empty}
    \emptyset = \widetilde{P_L,k}=\widetilde{P_M,k}
    \end{equation}
    where $\widetilde{P_L}_{,k}, \widetilde{P_M}_{,k}$ as in (\ref{slice}).
    Furthermore since $L,M\subseteq \ZZ^m \times \RR^{d-m}$, $\widetilde{P_L}$ and similiraly $\widetilde{P_M}$, has nonempty interior in $\ZZ^m \times \RR^{d-m}$. In other words, there exists a slice of $\widetilde{P_L}$ and similarily of $\widetilde{P_M}$ in $\RR^{d-m}$ with nonempty interior.
    \begin{figure}[H]
    \centering
    \begin{minipage}[b]{0.45\textwidth}
        \centering
        \includegraphics[width=\linewidth]{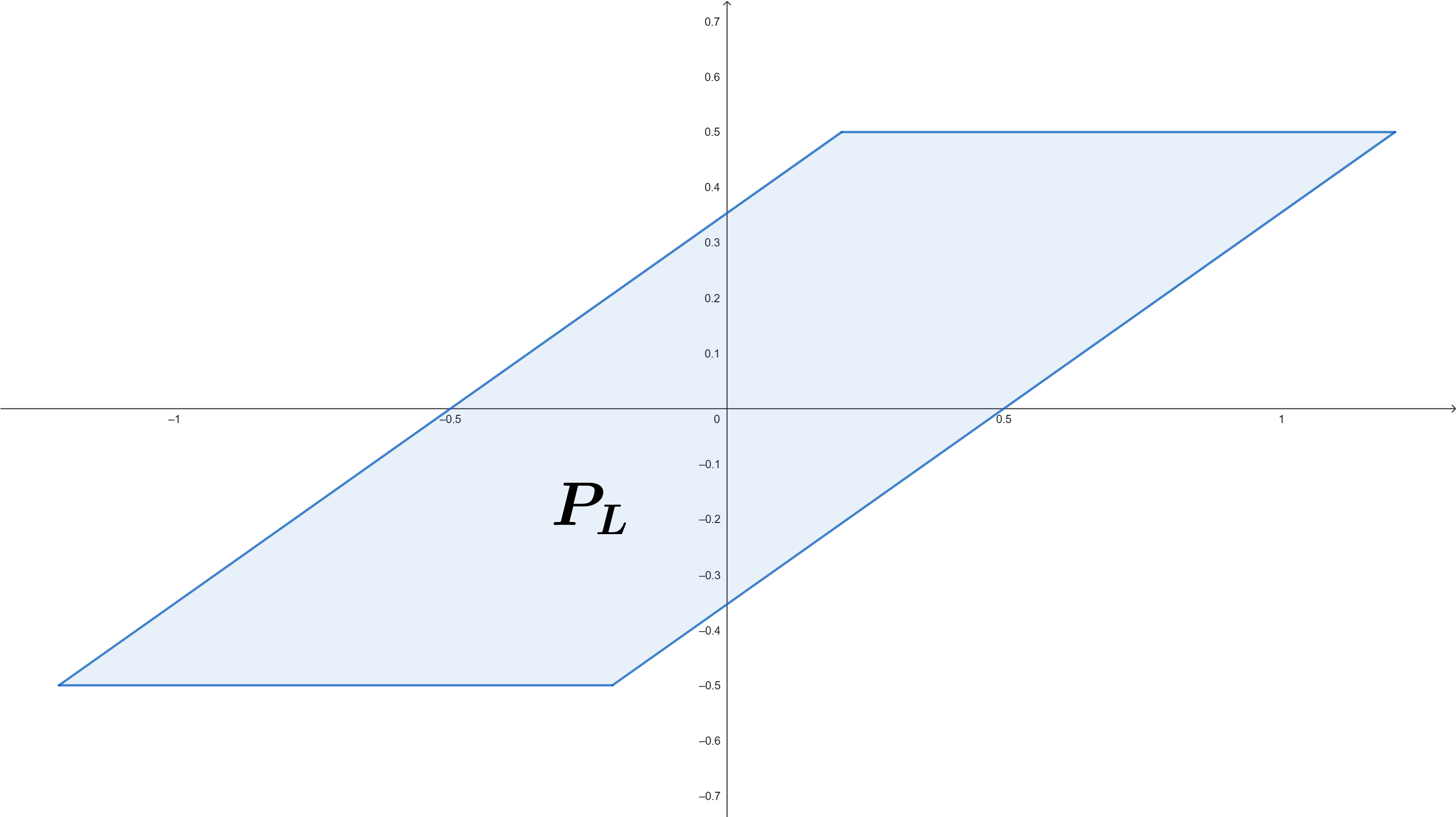}
        \caption{A fundamental parallelepipid of a lattice $L\subseteq \RR^2$}
    \end{minipage}\hfill
    \begin{minipage}[b]{0.45\textwidth}
        \centering
        \includegraphics[width=\linewidth]{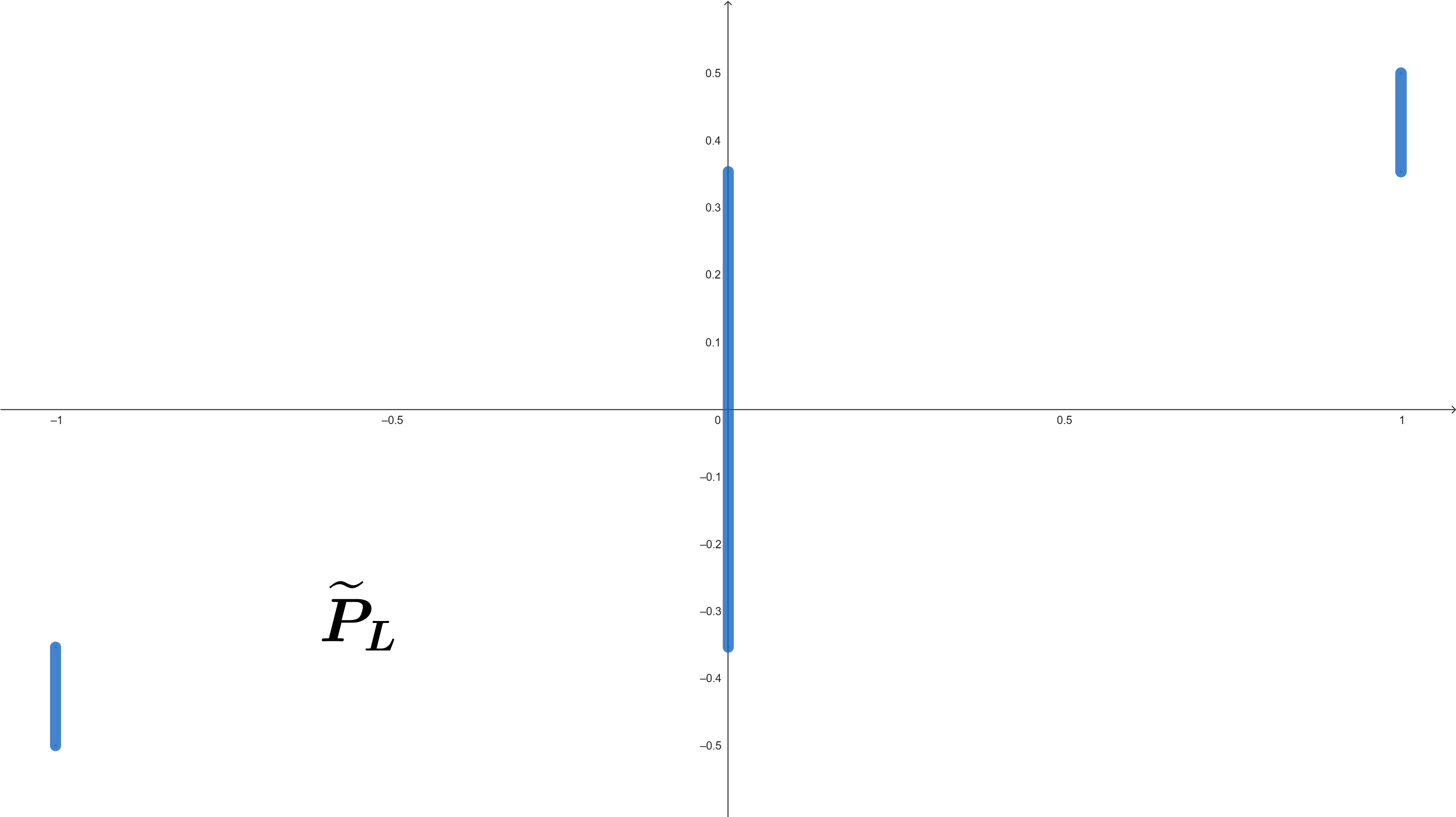}
        \caption{$\widetilde{P}_L=P_L\cap \ZZ \times \RR$ (in blue).}
    \end{minipage}
\end{figure}
     Indeed, the origin is an interior point of $P_L$ and $P_M$. This implies that both $P_L, P_M$ contain an open ball centered at the origin of radius $\epsilon>0$. The preimage of this open ball via the continuous map $\phi_0$ as in (\ref{phiz}), will then be a nonempty open set (the origin belongs to $\widetilde{P_L}$ and $\widetilde{P_M}$) contained in $\widetilde{P_L}$ and $\widetilde{P_M}$ and thus $\widetilde{P_L},\widetilde{P_M}$ have nonempty interior in $\ZZ^m \times \RR^{d-m}$. For the rest of this proof we will not refer to the original fundamental parallelepipeds $P_L,P_M$ on $\RR^d$. For this reason we denote $P_L:=\widetilde{P_L}$ and $P_M:=\widetilde{P_M}$.
    
    Take now $v_1,...,v_d$ linearly idependent elements of $L\subseteq \overline{L+M}= \ZZ^m \times \RR^{d-m}$, and take them large enough so that

    \begin{equation*}
        \Omega= \left \{\sum_{i\le d} t_i v_i \:| \: t_i \in \left[-\frac{1}{2}, \frac{1}{2}\right) \right \} \cap \ZZ ^m \times \RR^{d-m} \supseteq P_M, P_L
    \end{equation*}
    
    Consider $X= \Omega = \ZZ^m \times \RR^{d-m}/<v_1,...,v_d>_\ZZ$ and $G= (L+M)/ <v_1,...,v_d>_\ZZ$. $X$ as a bounded and measurable set (for the same reason that $P_L$ and $P_M$ are measurable)  of $\ZZ^m \times \RR^{d-m}$, can be viewed as a Borel probability space with the normalizeded $card \otimes m_{d-m}$ measure on $X$ and $G$ has a free pmp action on $X$. Also, $G$ is dense on $X$ and since $H=<v_1,..,v_d>_\ZZ$ is a full-rank sublattice of $L$, we have that
    \begin{equation*}
    	G= (L+M)/H \supseteq (M+H)/H \cong M/(H\cap M)
    \end{equation*}
    by the second isomorphism theorem. Notice that $ \rank(M/(H\cap M))>0$ for, assuming that $\rank(H\cap M) = d$ then $L+M$ would be a lattice in $\RR^d$ and so it can't be dense on $\ZZ^m \times \RR^{d-m}$ when $m<d$. Hence $G$ is an infinite group since $(M+H)/H$ is and so $\rank(G)>0$.

    Since $P_L, P_M$ are $L+M$ equidecomposable on $\ZZ^m \times \RR^{d-m}$, we have that $P_L, P_M$ are $G$-equidecomposable on $X$, and so to apply the theorem of Ciesla and Sabok Theorem \ref{CS22}, we need to show that $P_L, P_M$ are $G$ uniform sets in $X$. But we showed that $P_L, P_M$ are measurable sets on $X$, with nonempty interior, and $G$ is dense on $X$. By Lemma \ref{Guni}, $P_L, P_M$ are $ G$-uniform sets in $X$ and so an application of Ciesla and Sabok's Theorem \ref{CS22} gives us that $P_L, P _M$ are $L+M$ equidecomposable up to measure zero in $\ZZ^m \times \RR^{d-m}$ with measurable pieces.
    \begin{figure}[H]
    \centering
    \begin{minipage}[b]{0.45\textwidth}
        \centering
        \includegraphics[width=\linewidth]{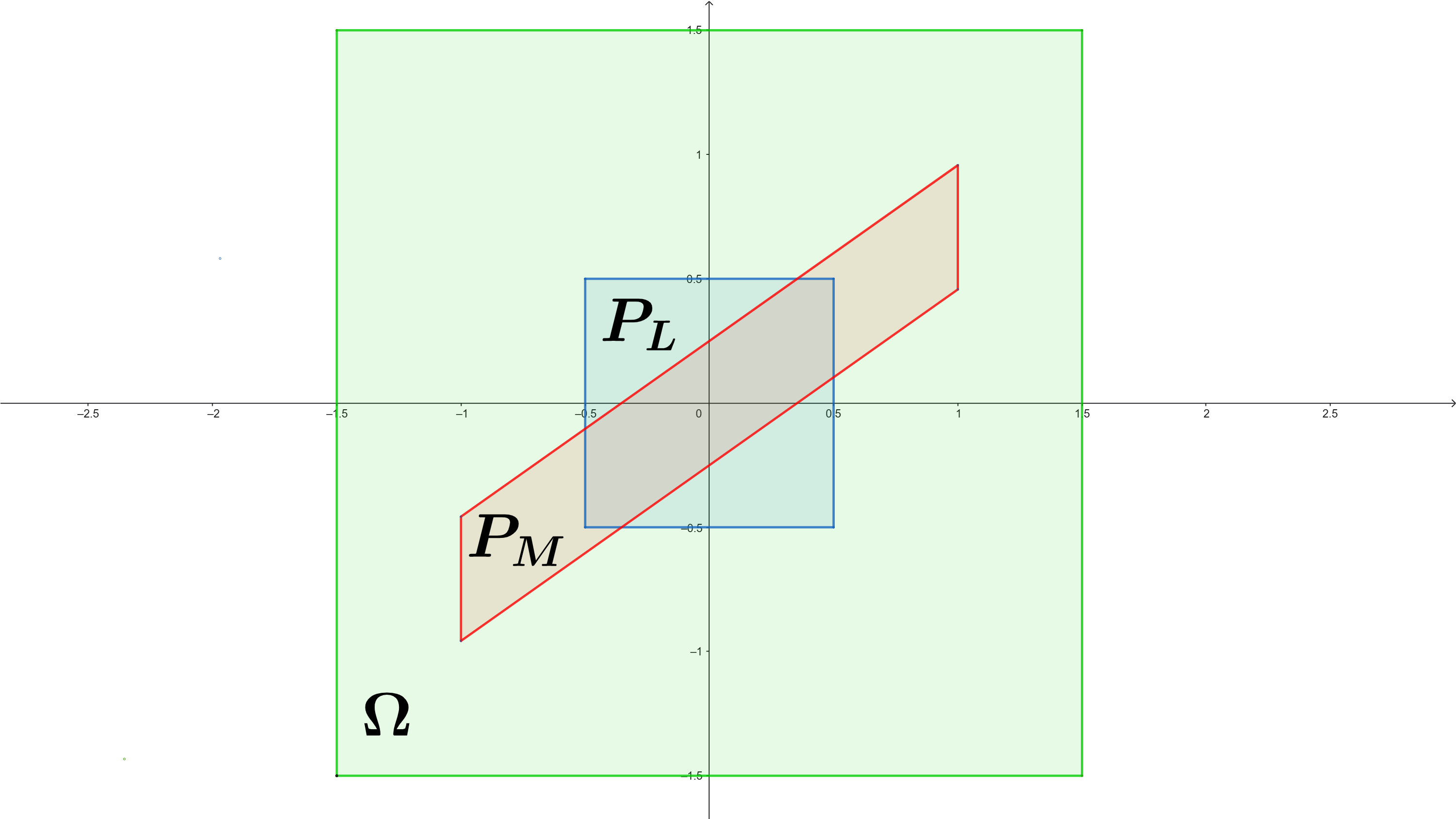}
        \caption{An example of $\Omega, P_L, P_M\subseteq \RR^2$ before intersecting with $\ZZ\times \RR$.}
        \label{10}
    \end{minipage}\hfill
    \begin{minipage}[b]{0.45\textwidth}
        \centering
        \includegraphics[width=\linewidth]{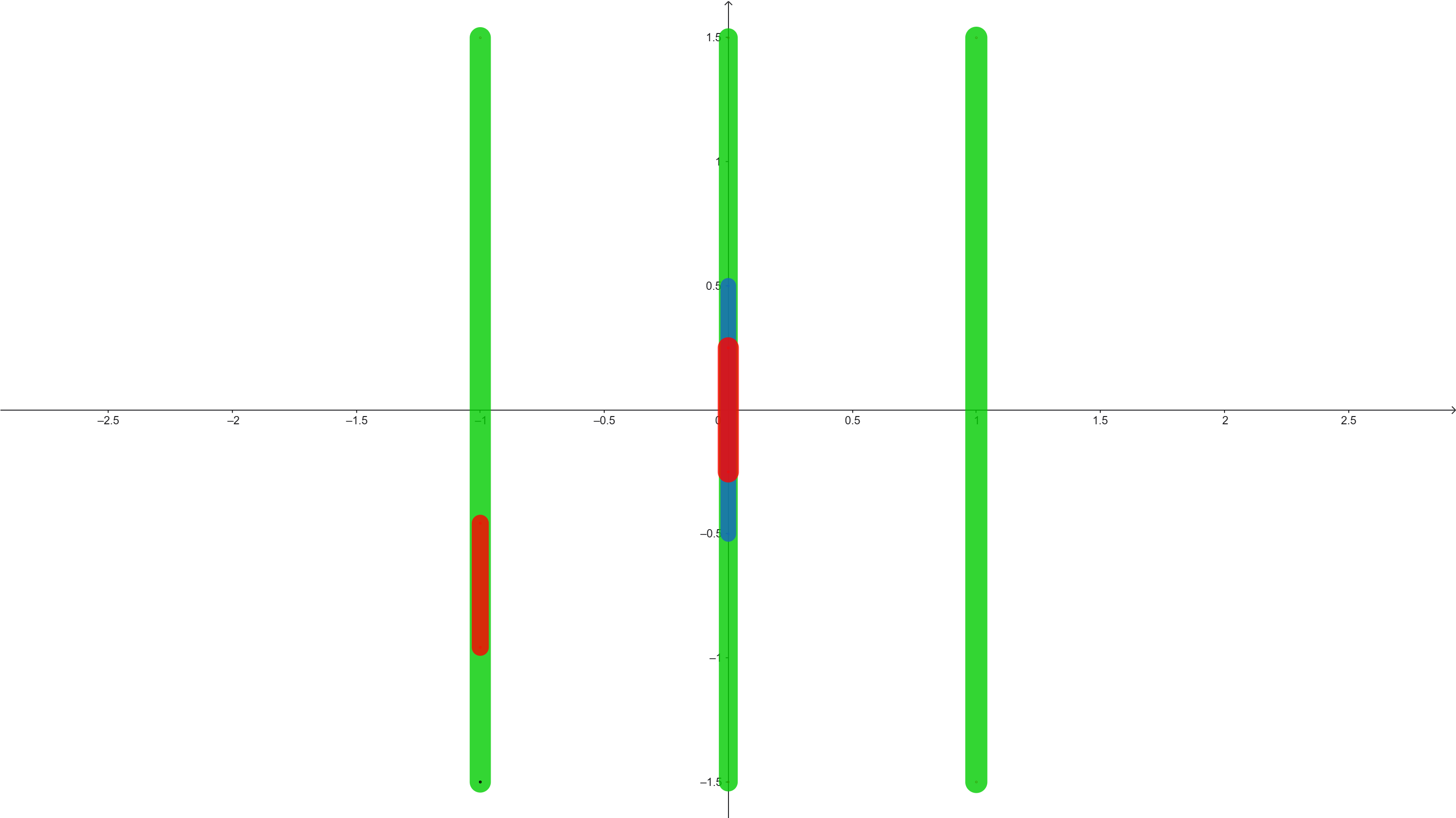}
        \caption{$\Omega$ (in green), $P_L$ (in blue) and  $P_M$ (in red) after intersecting with $\ZZ\times \RR$ in Figure \ref{10}.}
    \end{minipage}
\end{figure}

    This means that there exist a full measure set $X\subseteq \ZZ^m\times \RR^{d-m}$, $\mathcal{P}(\ZZ^{m})\otimes\mathcal{L}_{d-m}$- measurable pieces $S_i \subseteq P_L$ and elements $l_i \in L, m_i\in M$ for $i\le n\in \NN$ such that:
    \begin{equation*}
        P_L\cap X=\bigcup_{i\le n}S_i\cap X
    \end{equation*}
    is a partition of $P_L\cap X$ and:
    \begin{equation*}
        P_M \cap X = \bigcup_{i\le n} (S_i +l_i +m_i) \cap X
    \end{equation*}
    is a partition of $P_M\cap X$. Consider the set:
    \begin{equation*}
        F_1= \bigcup_{i\le n} (S_i +l_i) \cap X
    \end{equation*}
    We will show that the bounded and $\mathcal{P}(\ZZ^m)\times \mathcal{L}_{d-m}$-measurable $F_1$ is a common almost- fundamental domain of $\ZZ^m \times \RR^{d-m}$ i.e., it tiles $\ZZ^m\times \RR^{d-m}$ by translations of $L$ and by translations of $M$. We will show that $F_1$ tiles $\ZZ^m \times \RR^{d-m}$ by translations of $L$, and similarly, one can show the tiling by translations of $M$. In other words, we have to show that:
    \begin{equation*}
        \sum_{l\in L} 1_{F_1}(x-l)= 1
    \end{equation*}
    for $card\otimes m_{d-m}$-almost every $x \in \ZZ^m\times \RR^{d-m}$.
    Notice that $P_L$ is a measurable fundamental domain of $L$ in $\ZZ^m \times \RR^{d-m}$ and notice that since $X$ is a full measure (with respect to $card\times m_{d-m}$) set in $\ZZ^m \times \RR^{d-m}$ we have that:
    \begin{equation*}
        1_{P_L}=1_{P\cap X}= \sum_{i\le n} 1_{S_i\cap X}= \sum_{i\le n}1_{S_i}
    \end{equation*}
    for $card\otimes m_{d-m}$-a.e. on $\ZZ^m \times \RR^{d-m}$. As for $F_1$, since it is $L$- equidecomposable in $\ZZ^m \times \RR^{d-m}$ with $P_L$ we have that:
    \begin{equation*}
    1_{F_1}= \sum_{i\le n}1_{(S_i+l_i)\cap X} = \sum_{i\le n}1_{S_i+l_i}
    \end{equation*}
    for $card\otimes m_{d-m}$-a.e. on $\ZZ^m \times \RR^{d-m}$ and so:
    \begin{align*}
    	1 &= \sum_{l\in L} 1_{P_L}(x-l) 
    	= \sum_{l\in L} \sum_{i\le n} 1_{S_i}(x-l) 
    	= \sum_{i\le n} 1_{S_i+l_i}(x-(l+l_i)) = \\[2mm]
    	&= \sum_{i\le n} \sum_{l\in L} 1_{S_i + l_i}(x-l) 
    	= \sum_{l\in L} 1_{F_1}(x-l)
    \end{align*}
    for $card\otimes m_{d-m}$-a.e. on $\ZZ^m \times \RR^{d-m}$ as we had to show.
    
    Finally notice that if $A=\{x \in \ZZ^m \times \RR^{d-m} \: | \: \sum_{l\in L} 1_{F_1} (x-l)\neq 1\}$ is the exceptional set of $F_1$, it follows that every slice $A_k$ of $A$ is a $m_{d-m}$- null set as a subset of $\RR^{d-m}$ and hence the exceptional set of $F$ in $\RR^d$
    \begin{equation*}
        [0,1)^m \times \{0\}^{d-m} + A= \bigcup_{k\in \ZZ^m} ([0,1)^m+k)\times A_k
    \end{equation*}
    is $m_d= m_m\otimes \: m_{d-m}$- null set and so $F$ is a bounded common almost-fundamental domain of $L,M$ in $\RR^d$. By Theorem \ref{MeasurableFD} $L,M$ admit a bounded and measurable common fundamental domain in $\RR^d$ as we had to show and the proof is complete.
    \end{itemize}    
\end{proof}
In \cite{GKS25} (Theorem 1.1) it was proved that for any two lattices whose sum is direct, it is impossible to find even a non-measurable bounded common fundamental domain. An immediate application of Theorem \ref{BMF} generalize this result for any two lattices.
\begin{corollary}\label{Evol}
Any two full-rank lattices admitting a bounded, common fundamental domain have essentially the same volume.
\end{corollary}
\begin{proof}
    Assume that $L, M\subseteq\RR^d$ are full-rank lattices and $F_1\subseteq \RR^d$ is a bounded common fundamental domain of $L, M$ in $\RR^d$. By Theorem $\ref{BMF}$, there exists a bounded and measurable common fundamental domain $F\subseteq \RR^d$ of $ L, M$ in $\RR^d$. Denote by $m_d$ the Lebesgue measure on $\RR^d$. Since $F$ is a measurable fundamental domain of $L$, we must have $m_d(F)= \vol(L)$. Symmetrically, since $F$ is a measurable fundamental domain of $M$, we also must have that $m_d(F)=\vol(M)$ and thus $\vol(L)=\vol(M)$.
\end{proof}
We can now prove our main result:
\begin{proof}[Proof of Theorem \ref{MT}]
Let $L, M\subseteq \RR^d$ be two full-rank lattices of equal volumes. By Theorem \ref{BF}, there exists a bounded common fundamental domain of $L, M$ in $\RR^d$, which can be chosen measurable by Theorem \ref{BMF}.
\end{proof}
\bibliographystyle{amsalpha}  
\bibliography{references}
\end{document}